\documentclass[12pt,reqno]{amsart}
\usepackage{latexsym}
\usepackage{amssymb}
\usepackage{mathrsfs}
\usepackage{amsmath}
\usepackage{fancybox,color}
\usepackage{enumerate}
\usepackage{enumitem}
\usepackage[latin1]{inputenc}
\usepackage{soul} 
\usepackage[colorlinks=true, linkcolor=magenta, citecolor=magenta]{hyperref}

\usepackage{comment}

\makeatletter 
\renewcommand\subsection{\@startsection{subsection}{2}%
  \z@{-.5\linespacing\@plus-.7\linespacing}{.3\linespacing}%
  {\normalfont\bfseries}}
\makeatother

\usepackage[left=2.15cm,top=2.41cm,right=2.15cm]{geometry}

\def\1{\raisebox{2pt}{\rm{$\chi$}}}

\newtheorem{theorem}{Theorem}[section]
\newtheorem{corollary}[theorem]{Corollary}
\newtheorem{lemma}[theorem]{Lemma}
\newtheorem{proposition}[theorem]{Proposition}

\theoremstyle{definition}
\newtheorem{definition}[theorem]{Definition}
\newtheorem{remark}[theorem]{Remark}

\newtheorem{example}[theorem]{Example}

\newcommand{\eps}{{\varepsilon}}
\newcommand{\R}{{\mathbb R}}

\newcommand{\B}{{\mathbb B}}
\newcommand{\N}{{\mathbb N}}

\newcommand\diam{\operatorname{diam}}

\DeclareMathOperator*{\esssup}{ess\, sup}

\def\1{\raisebox{2pt}{\rm{$\chi$}}}

\newcommand{\Lip}{\operatorname{Lip}}

\def\vint_#1{\mathchoice%
        {\mathop{\kern 0.2em\vrule width 0.6em height 0.69678ex depth -0.58065ex
                \kern -0.8em \intop}\nolimits_{\kern -0.4em#1}}%
        {\mathop{\kern 0.1em\vrule width 0.5em height 0.69678ex depth -0.60387ex
                \kern -0.6em \intop}\nolimits_{#1}}%
        {\mathop{\kern 0.1em\vrule width 0.5em height 0.69678ex
            depth -0.60387ex
                \kern -0.6em \intop}\nolimits_{#1}}%
        {\mathop{\kern 0.1em\vrule width 0.5em height 0.69678ex depth -0.60387ex
                \kern -0.6em \intop}\nolimits_{#1}}}
\def\vintslides_#1{\mathchoice%
        {\mathop{\kern 0.1em\vrule width 0.5em height 0.697ex depth -0.581ex
                \kern -0.6em \intop}\nolimits_{\kern -0.4em#1}}%
        {\mathop{\kern 0.1em\vrule width 0.3em height 0.697ex depth -0.604ex
                \kern -0.4em \intop}\nolimits_{#1}}%
        {\mathop{\kern 0.1em\vrule width 0.3em height 0.697ex depth -0.604ex
                \kern -0.4em \intop}\nolimits_{#1}}%
        {\mathop{\kern 0.1em\vrule width 0.3em height 0.697ex depth -0.604ex
                \kern -0.4em \intop}\nolimits_{#1}}}

\newcommand{\aveint}[2]{\mathchoice%
        {\mathop{\kern 0.2em\vrule width 0.6em height 0.69678ex depth -0.58065ex
                \kern -0.8em \intop}\nolimits_{\kern -0.45em#1}^{#2}}%
        {\mathop{\kern 0.1em\vrule width 0.5em height 0.69678ex depth -0.60387ex
                \kern -0.6em \intop}\nolimits_{#1}^{#2}}%
        {\mathop{\kern 0.1em\vrule width 0.5em height 0.69678ex depth -0.60387ex
                \kern -0.6em \intop}\nolimits_{#1}^{#2}}%
        {\mathop{\kern 0.1em\vrule width 0.5em height 0.69678ex depth -0.60387ex
                \kern -0.6em \intop}\nolimits_{#1}^{#2}}}

\newcommand{\dist}{\operatorname{dist}}

\title[A maximal function approach to Poincar\'e inequalities]{A maximal function approach to two-measure Poincar\'e inequalities} 

\author[J.\! Kinnunen]{Juha Kinnunen}   
\address[J.K.]{Department of Mathematics, Aalto University, P.O. Box 11100, FI-00076 Aalto University, Finland}
\email{juha.k.kinnunen@aalto.fi}

\author[R.\! Korte]{Riikka Korte}   
\address[R.K.]{Department of Mathematics, Aalto University, P.O. Box 11100, FI-00076 Aalto University, Finland}
\email{riikka.korte@aalto.fi}

\author[J.\! Lehrb\"ack]{Juha Lehrb\"ack}   
\address[J.L.]{University of Jyvaskyla, Department of Mathematics and Statistics, P.O. Box 35, FI-40014 University of Jyvaskyla, Finland}
\email{juha.lehrback@jyu.fi}

\author[A.V.\! V\"ah\"akangas]{Antti V. V\"ah\"akangas}
\address[A.V.V.]{University of Jyvaskyla, Department of Mathematics and Statistics, P.O. Box 35, FI-40014 University of Jyvaskyla, Finland} 
\email{antti.vahakangas@iki.fi}


\thanks{This research was supported by the Academy of Finland.}

\begin{document}

\begin{abstract}
This paper extends the self-improvement result of Keith and Zhong in~\cite{MR2415381} to the two-measure case.
Our main result shows that a two-measure $(p,p)$-Poincar\'e inequality for $1<p<\infty$ improves to a $(p,p-\eps)$-Poincar\'e inequality for some $\eps>0$ under a balance condition on the measures. 
The corresponding result for a maximal Poincar\'e inequality is also considered. 
In this case the left-hand side in the Poincar\'e inequality is replaced with an integral of a  sharp maximal function and  the results hold without  a balance condition.
Moreover,  validity of maximal Poincar\'e inequalities is used to characterize the self-improvement of two-measure Poincar\'e inequalities.
Examples are constructed to illustrate the role of the assumptions.
Harmonic analysis and PDE  techniques are used extensively in the arguments.
\end{abstract}

\maketitle

\section{Introduction}

Let $X=(X,d,\nu,\mu)$ be a metric space equipped with two Borel measures $\mu$ and $\nu$, and
let $1\le q,p<\infty$. In this work we are interested in properties of the
\emph{two-measure $(q,p)$-Poincar\'e inequalities}
\begin{equation}\label{eq.two_meas_intro}
\biggl(\vint_{B} \lvert u(x)-u_{B;\nu}\rvert^q\,d\nu(x)\biggr)^{1/q}
\le C\diam(B)\biggl(\vint_{B} g(x)^p \,d\mu(x)\biggr)^{1/p}\,.
\end{equation}
We say that the space $X$ supports a two-measure $(q,p)$-Poincar\'e inequality,
if there is a constant $C>0$ such that inequality~\eqref{eq.two_meas_intro} holds for all 
balls $B$ in $X$ whenever $u$ is a Lipschitz function in $X$ and $g$ is 
 a $p$-weak  upper gradient of $u$;
see  Sections~\ref{s.prelim} and~\ref{s.two_measure_PI}  for the relevant definitions.

An interesting feature of these inequalities is that they are often self-improving: 
a $(q,p)$-Poincar\'e inequality implies a similar inequality for other values of the parameters $p$ and $q$.
By H\"older's inequality, we can  increase $p$ and decrease $q$. Thus the actual
self-improvement concerns the opposite directions. Next we recall some of the known results.

In the one-measure case $\mu=\nu$, Haj\l asz and Koskela showed in~\cite{MR1336257} that if $\mu$
is doubling and $X$ supports a $(1,p)$-Poincar\'e inequality, then there exists $q_0>p$
such that $X$ supports a $(q,p)$-Poincar\'e inequality for every $1\le q\le q_0$. The other direction,
which is more delicate, was settled by Keith and Zhong in~\cite{MR2415381}, where they proved that 
if $X$ is complete, $\mu$ is doubling, and $X$ supports a $(1,p)$-Poincar\'e inequality for some $1<p<\infty$, 
then there exists $\eps_0>0$
such that $X$ supports a $(1,p-\eps)$-Poincar\'e inequality for $0<\eps\le\eps_0$. 
In the scale of Lipschitz functions, the proof of~\cite{MR2415381} also works
in non-complete geodesic spaces.  Recently, new proofs and extensions for  the Keith--Zhong result  
have been given in~\cite{DeJarnette,SEB,KinnunenLehrbackVahakangasZhong}.
In many respects this paper is a continuation of the work initiated in~\cite{KinnunenLehrbackVahakangasZhong}. 

In the two-measure case, the improvement on the left-hand side of~\eqref{eq.two_meas_intro} 
follows from the results that have been 
established in various settings in a series of papers
by Franchi, MacManus, P\'erez, and Wheeden~\cite{MR1609261,MR1604816,MR2027891}.
 These works also discuss the question how to obtain weighted Poincar\'e inequalities from non-weighted 
inequalities. 
A particular consequence of the results in~\cite{MR1604816} is that if
$\mu$ and $\nu$ are doubling measures and $X$ supports a two-measure $(1,p)$-Poincar\'e inequality,
and moreover for some $q \ge p$ the measures $\mu$ and $\nu$ satisfy the \emph{balance condition} 
\begin{equation}\label{eq.balance_intro}
\frac{\diam(B')}{\diam(B)} \biggl(\frac{\nu(B')}{\nu(B)}\biggr)^{1/q}\le
C\biggl(\frac{\mu(B')}{\mu(B)}\biggr)^{1/p}
\end{equation}
whenever the balls $B$ and $B'=B(x',r')$ are such that $x'\in B$ and $0<r'\le \diam(B)$,
then $X$ supports also a two-measure $(q,p)$-Poincar\'e inequality, but possibly with
slightly larger balls on the right-hand side of~\eqref{eq.two_meas_intro}. 
The above balance condition was introduced and applied by Chanillo and Wheeden~\cite{ChanilloWheeden1985}
in connection with two-weight Poincar\'e and Sobolev inequalities in the Euclidean space $\R^n$.
Subsequently, this  condition  has appeared, for instance, in~\cite{ChanilloWheeden1985CPDE, FranchiLuWheeden1995, MR1189903, MR1816566}.

Our purpose in this paper is to study the 
self-improvement with respect to the right-hand side of~\eqref{eq.two_meas_intro} 
in a geodesic metric space $X$ equipped with two measures $\mu$ and $\nu$.
More precisely, we start with a two-measure $(p,p)$-Poincar\'e inequality for $1<p<\infty$, 
and improve this into a $(p,p-\eps)$-Poincar\'e inequality for some $\eps>0$, under certain additional conditions. 
To some extent, results in this direction could be obtained
by combining the Keith--Zhong result on one-measure inequalities with the abstract weighted
machinery  in~\cite{MR1609261,MR1604816,MR2027891}, but such a combination of two
extensive theories easily distracts from the essential mechanisms behind the self-improvement,
and it is also difficult to analyze dependencies of the relevant parameters.
We propose a direct approach, 
where we use the assumed $(p,p)$-Poincar\'e inequality only once,
and therefore our proof can be, for instance, used to track down a reasonable estimate for the constant
in the resulting $(p,p-\eps)$-Poincar\'e inequality.
Moreover, the direct examination of the two-measure setting
reveals several interesting new phenomena that are not  clearly visible in the one-measure case. 

The first new feature is that the balance condition~\eqref{eq.balance_intro}, which is necessary for the validity of
the two-measure $(q,p)$-Poincar\'e inequality~\eqref{eq.two_meas_intro} by Lemma~\ref{l.Poincare_implies_balance},
does not self-improve; see Example~\ref{ex.bal_no_impro}.  
This poses an additional restriction for the self-improvement of two-measure Poincar\'e inequalities, and
 in Example~\ref{ex.no_impro} we describe a situation where our other assumptions are satisfied, 
but a two-measure $(p,p)$-Poincar\'e inequality does not improve to a $(p,p-\eps)$-Poincar\'e inequality for any $\eps>0$,
since the $(p,p-\eps)$-balance condition does not hold for any $\eps>0$. To obtain a better two-measure Poincar\'e inequality
we thus need to assume an \emph{a priori} stronger balance condition. As it turns out,  a slight improvement given by 
a suitable \emph{bumbed balance condition}, introduced by Lerner and P\'erez in~\cite{MR2375698}
for the Muckenhoupt weights, is sufficient for the improvement of
two-measure Poincar\'e inequalities; see Definition~\ref{d.bumped_balance_def}, Theorem~\ref{t.bump_impro}, 
and Theorem~\ref{t.main_thm}.
In fact, many of our results have  counterparts
for  Muckenhoupt weights; see 
Lerner and P\'erez~\cite{MR2375698}.

Another new feature in our approach is the  introduction of the so-called \emph{maximal Poincar\'e inequalities}, in which the left-hand side 
of~\eqref{eq.two_meas_intro} is replaced with an integral of a sharp maximal function.
In the one-measure case,
the corresponding maximal Poincar\'e inequalities are essentially equivalent to 
the usual Poincar\'e inequalities; in fact, the maximal Poincar\'e inequalities
have been used as a tool in the proofs of the self-improvement results, for instance,
in~\cite{MR2415381,KinnunenLehrbackVahakangasZhong}. 
However, there is a difference  between the usual and maximal Poincar\'e inequalities in the two-measure case.
More precisely, the maximal Poincar\'e inequalities often enjoy certain self-improvement  independent of
the balance conditions; see Theorems~\ref{t.global_epsilon_0} and~\ref{t.global_improvement} for details. 
This shows that the maximal Poincar\'e inequalities are strictly stronger than the usual two-measure Poincar\'e inequalities;
cf.\ Example~\ref{ex.sharp_explanation}. 
Moreover, the validity of maximal Poincar\'e inequalities can be used to characterize
the self-improvement of two-measure Poincar\'e inequalities; see Corollary~\ref{c.self_impro}
and Theorem~\ref{t.global_epsilon_0}.

The outline of the paper is as follows.
In Section~\ref{s.prelim} we  recall preliminaries related to (geodesic)
metric two-measure spaces and Muckenhoupt weights. In Section~\ref{s.balance} we introduce
the $(q,p)$-balance condition and the bumped version of the $(p,p)$-balance condition.
We also establish some basic relations between the balance conditions for different values of
the parameters $q$ and $p$ (Proposition~\ref{p.basic_balance}) and show that, under mild conditions, 
the bumped $(p,p)$-balance condition
is equivalent to a $(p-\eps,p-\eps)$-balance condition for some $\eps>0$ (Theorem~\ref{t.bump_impro}).
This section is concluded with an example showing that the balance conditions do not always self-improve. 

Usual two-measure Poincar\'e inequalities are introduced in Section~\ref{s.two_measure_PI}, where
we also prove the necessity of the $(q,p)$-balance condition for the two-measure $(q,p)$-Poincar\'e inequality.
In Section~\ref{s.max_PI} we define both the sharp maximal functions related to families of balls and the
associated maximal Poincar\'e inequalities, and we also study the relation between usual and maximal
Poincar\'e inequalities (Lemma~\ref{l.straightforward}). Section~\ref{s.mainR} then contains the statements of
our main results: first Theorem~\ref{t.main_thm} and Corollary~\ref{c.self_impro} concerning the self-improvement
of the two-measure $(p,p)$-Poincar\'e inequality, and then Theorem~\ref{t.global_epsilon_0}, which shows that 
maximal Poincar\'e inequalities with respect to a certain global maximal function always self-improve.
The latter theorem also creates a link between the self-improvement of usual and maximal
Poincar\'e inequalities. Besides the assumptions that $X$ is geodesic and the measures
$\mu$ and $\nu$ satisfy relevant balance conditions, in the main results of Section~\ref{s.mainR} we assume
that $\nu$ is an $A_\infty(\mu)$ weighted measure (Definition~\ref{d.comparability})
and the space $X$ satisfies an independence property for the upper gradients (Definition~\ref{d.independence}).

The outlines of the proofs of our main results are given  in Section~\ref{s.mainR},
but these proofs rely on technical tools that are postponed to the final sections of the  paper. First, in Section~\ref{s.main},
we establish Theorem~\ref{t.main_local}, which plays an important role in the proof of Theorem~\ref{t.main_thm}.
This part is based on two-measure adaptations of the ideas from~\cite{KinnunenLehrbackVahakangasZhong}, but due to the
subtle modifications that are needed we present most of the details. Finally, in Sections~\ref{s.Whitney_ext} and~\ref{s.global} 
we conclude the proof of Theorem~\ref{t.global_epsilon_0}. The main result needed for
Theorem~\ref{t.global_epsilon_0} is Theorem~\ref{t.global_improvement}. In the proof of the latter theorem we
need a somewhat curious Lipschitz extension, which does not \emph{decrease} the global maximal function too much. 
Such an extension is constructed in Section~\ref{s.Whitney_ext}. The idea is to first take the usual Whitney
extention, and then modify  it by adding suitable bumb functions which guarantee that the sharp maximal function
of the modified extension is large enough. Theorem~\ref{t.global_improvement}, which actually contains a stronger
version of the most important implication in Theorem~\ref{t.global_epsilon_0}, is then stated and proved
in the final Section~\ref{s.global}.

\begin{remark}[Tracking constants]
The letter $C$ is used to denote 
positive constants, whose dependencies can vary and whose value
can change from one occurrence to another. Some of our 
self-improvement  results are based on 
quantitative estimates and absorption arguments, where
it is often crucial to track the dependencies of constants more carefully.
For this purpose, we will use the following notational convention:
$C({\ast,\dotsb,\ast})$ denotes a positive constant which 
depends at most on the parameters indicated by the $\ast$'s but whose actual
value can change from one occurrence to another, even within a single line.
\end{remark}


\section{Preliminaries}\label{s.prelim}

\subsection{Metric two-measure spaces}\label{s.metric}

We assume that 
$X=(X,d,\nu,\mu)$ is a {\em metric two-measure space} 
equipped with a metric $d$ and 
{\em two} positive complete Borel
measures $\nu$ and $\mu$, and satisfying $\# X\ge 2$. 
We also assume  throughout this paper  that  
\begin{equation}\label{e.finite_on_balls}
0<\nu(B)<\infty\quad \text{ and } \quad 0<\mu(B)<\infty
\end{equation}
for all (open) balls 
\[B=B(x,r)=\{y\in X\,:\, d(y,x)<r\}\subset X\] 
with $x\in X$ and $r>0$, 
and that the measures $\nu$ and $\mu$ are {\em doubling}, that is,
there are constants $c_\nu,c_\mu> 1$ such that
\begin{equation}\label{e.doubling}
\nu(2B) \le c_\nu\, \nu(B)\quad\text{ and }\quad \mu(2B) \le c_\mu\, \mu(B)
\end{equation}
for all balls $B=B(x,r)$ in $X$. 
Here we use for $0<t<\infty$ the notation $tB=B(x,tr)$. 
We remark that $X$ is separable
under these assumptions, see \cite[Proposition 1.6]{MR2867756}.

Iteration of the doubling condition \eqref{e.doubling} for the measure $\nu$  shows  that
\begin{equation}\label{e.radius_measure}
\frac{\nu(B')}{\nu(B)}\ge 2^{-s}\biggl(\frac{r'}{\diam(B)}\biggr)^s\,,\qquad s=\log_2 c_\nu>0\,,
\end{equation}
whenever $B$ and  $B'=B(x',r')$ are balls
in $X$ such that $x'\in B$ and $r'\le \diam(B)${;} see 
for instance~\cite[p.~31]{MR1800917}
and~\cite[Lemma 3.3]{MR2867756}. 
The corresponding estimate holds for $\mu$, as well.

When $A\subset X$, we let $\mathbf{1}_{A}$
denote  the characteristic function of $A$; that is, $\mathbf{1}_{A}(x)=1$ if $x\in A$
and $\mathbf{1}_{A}(x)=0$ if $x\in X\setminus A$. 
We use the notation 
\[
u_{A;\nu}=\vint_{A} u(y)\,d\nu(y)=\frac{1}{\nu(A)}\int_A u(y)\,d\nu(y)
\]
for the integral average of $u\in L^1(A;d\nu)$ in a Borel set $A\subset X$
with $0<\nu(A)<\infty$. 
If $1\le p<\infty$ and $u\colon X\to \R$ is a $\mu$-measurable function, then $u\in L^p_{\textup{loc}}(X;d\mu)$ 
means that for each $x\in X$ there exists $r_x>0$ such that $u\in L^p(B(x,r_x);d\mu)$, 
i.e., $\int_{B(x,r_x)} \lvert u(y)\rvert^p\,d\mu(y)<\infty$.

\subsection{Geodesic two-measure spaces}\label{s.geodesic}
Let $X=(X,d,\nu,\mu)$ be a metric two-measure space, satisfying the assumptions in Section~\ref{s.metric}.
By a {\em curve} we mean a nonconstant, rectifiable, and continuous
mapping from a compact interval of $\R$ to $X$; we tacitly assume
that all curves are parametrized by their arc-length.
We say that $X$ is a {\em geodesic two-measure space}, if 
any two distinct points in $X$
can be joined by a curve whose length is equal to the distance between the two points. 

A geodesic two-measure space $X$  is connected, and therefore it 
holds for all balls $B$ in $X$ that 
\begin{equation}\label{e.diams}
0<\diam(2B)\le 4\diam(B)\,.
\end{equation}
Moreover, by the connectedness, there are constants $C>0$ and $\sigma>0$ such that
\begin{equation}\label{e.rev_d}
\frac{\nu(B')}{\nu(B)}\le C\biggl(\frac{r'}{\diam(B)}\biggr)^{\sigma}
\end{equation}
whenever $B$ and $B'=B(x',r')$ are balls in $X$ such that $x'\in B$ and $r'\le \diam(B)$.
Again, a corresponding inequality holds for the measure $\mu$ as well.
For the proof of inequality \eqref{e.rev_d} we refer to \cite[Corollary~3.8]{MR2867756}.

The following lemma is \cite[Lemma 12.1.2]{MR3363168}.

\begin{lemma}\label{l.continuous}
Assume that $X$ is a geodesic two-measure space and that $A\subset X$ is a $\nu$-measurable set. 
Then the function
\[
r\mapsto \frac{\nu(B(x,r)\cap A)}{\nu(B(x,r))}\,:\, (0,\infty)\to \R
\]
is continuous  for all 
$x\in X$.
\end{lemma}


The following lemma, in turn, is \cite[Lemma 2.5]{KinnunenLehrbackVahakangasZhong}.

\begin{lemma}\label{l.ball_measures}
Assume that $B$ and $B'=B(x',r')$ are balls in a geodesic two-measure space $X$ such
that $x'\in B$ and $0<r'\le \diam(B)$. Then
$\nu(B')\le c_\nu^3 \nu(B'\cap B)$.
\end{lemma}

\subsection{Muckenhoupt weights and weighted measures}\label{s.Muckenhoupt}

Assume that $X=(X,d,\nu,\mu)$ is a geodesic two-measure space.
Let  $s'>0$ and $\sigma'=\sigma>0$ 
be the exponents as in \eqref{e.radius_measure} and \eqref{e.rev_d} 
for the measures $\mu$ and $\nu$, respectively.
It follows that there is a constant $C>0$ such that
\begin{equation}\label{e.nu_mu_comparison}
\frac{\nu(B')}{\nu(B)}\leq C\left(\frac{r'}{\diam(B)}\right)^{\sigma'}\leq C\left(\frac{\mu(B')}{\mu(B)}\right)^{\sigma'/s'}\,,
\end{equation}
whenever $B$ and $B'=B(x',r')$ are balls in $X$ such that $x'\in B$ and $0<r'\le \diam(B)$. 
In our main results, we also need the following stronger version of this estimate.

\begin{definition}\label{d.comparability}
We say that $\nu$ is an {\em $A_\infty(\mu)$-weighted measure}
if there exist constants $c_{\nu,\mu}>0$ and $\delta>0$ such that
inequality
\[
\frac{\nu(A)}{\nu(B)} \le c_{\nu,\mu} \biggl(\frac{\mu(A)}{\mu(B)}\biggr)^\delta
\]
holds whenever $B\subset X$ is a ball and $A\subset B$ is a Borel set.
\end{definition}

Let us justify the terminology that is used in Definition \ref{d.comparability}.
If $\nu$ is  an $A_\infty(\mu)$-weighted measure,
then  $\nu$ is absolutely continuous with respect to $\mu$ and
both of these measures are $\sigma$-finite
by \eqref{e.finite_on_balls}. By the Radon--Nikodym Theorem, there
exists a non-negative Borel function $w\colon X\to \R$ such that 
\[
\nu(A)=\int_A w(x)\,d\mu(x)
\]
for all Borel sets $A\subset X$, 
 and so $w$ belongs  to the so-called Muckenhoupt class $A_\infty(\mu)$
 in the sense of the following standard Definition \ref{d.a_infty}.

A Borel function $w\colon X\to \R$ satisfying
$w(x)>0$ for $\mu$-almost every $x\in X$ and
$\int_B w\,d\mu<\infty$ for all balls $B\subset X$
is called a {\em weight.}
We write 
$w(A)=\int_A w\,d\mu$ if $A\subset X$ is a Borel set
and $w$ is a weight.

\begin{definition}\label{d.a_infty}
A weight $w$ belongs  to the {\em Muckenhoupt class $A_\infty(\mu)$},  denoted  $w\in A_\infty(\mu)$,
if there are constants $C>0$ and $\delta>0$ such that
\[
\frac{w(A)}{w(B)} \le C \biggl(\frac{\mu(A)}{\mu(B)}\biggr)^\delta 
\]
whenever  $B\subset X$ is a ball and $A\subset B$ is a Borel set.  
\end{definition}

We also need the corresponding classes for exponents $1\le p<\infty$.

\begin{definition}
Let $1\le p<\infty$.
A weight $w$ belongs to the {\em Muckenhoupt class $A_p(\mu)$}, denoted  $w\in A_p(\mu)$, if
there is a constant $A>0$ such that, for every ball $B$ in $X$, 
\begin{equation}\label{a_p}
\biggl( \vint_B w\,d\mu\biggr) \biggl(\vint_B w^{-1/(p-1)}\,d\mu\biggr)^{p-1} \le A\,
\quad\text{ if } p>1\,,
\end{equation}
and
\begin{equation}\label{a_1}
\biggl( \vint_B w\,d\mu\biggr) \esssup_{y\in B} \frac{1}{w(y)} \le A\,
\quad\text{ if } p=1\,.
\end{equation}
\end{definition}

By \cite[Chapter I, Theorem 15]{MR1011673}, it holds for 
every $1<p<q<\infty$ that
\begin{equation}\label{ap_relations}
A_1(\mu)\subset A_p(\mu)\subset A_q(\mu)\subset A_\infty(\mu).
\end{equation}
Furthermore, the equality 
$A_\infty(\mu)=\bigcup_{1\le p<\infty} A_p(\mu)$
is valid under our standing assumptions
since the measure $\mu$ is doubling and
$\mu(B(x,r))$ increases
continuously with $r$ for each $x\in X$.  The 
latter property follows from the assumption that $X$ is  geodesic; we  
refer to \cite[Chapter~I, Theorem~18]{MR1011673}
 and  \cite[Proposition~11.5.3]{MR3363168} for details.

\subsection{Lipschitz functions}

Let $A\subset X$ and $0\le \kappa <\infty$. 
We say that a function
$u\colon A\to \R$  is  {\em $\kappa$-Lipschitz}, if
\[
\lvert u(x)-u(y)\rvert\le \kappa\, d(x,y)\qquad \text{ for all } x,y\in A\,.
\]
If $u\colon A\to \R$ is  $\kappa$-Lipschitz,  then the classical McShane extension
\begin{equation}\label{McShane}
v(x)=\inf \{ u(y) + \kappa \,d(x,y)\,:\,y\in A\}\,,\qquad x\in X\,,
\end{equation}
defines a $\kappa$-Lipschitz function $v\colon X\to \R$, which satisfies
$v|_A = u$; we refer to \cite[pp.~43--44]{MR1800917}.

The set of all Lipschitz functions $u\colon A\to\R$
is denoted by $\Lip(A)$. 
That is, we have $u\in \Lip(A)$ if, and only if, $u\colon A\to \R$ is $\kappa$-Lipschitz for some $0\le \kappa<\infty$.

We also say that a function $u\colon X\to \R$ has  {\em bounded support}, if
the set $\overline{\{x\in X\,:\,u(x)\not=0\}}$ is contained in some ball $B$ in $X$.

\section{Balance conditions}\label{s.balance}

The following balance condition  for measures was introduced in~\cite{ChanilloWheeden1985}. It is  closely  related 
to the  two-measure  Poincar\'e inequalities that are  discussed in
Section~\ref{s.two_measure_PI}.

\begin{definition}\label{d.qp_balance}
Let $1\le q,p<\infty$. 
We say that a metric (or a geodesic) two-measure space
$(X,d,\nu,\mu)$ {\em satisfies a $(q,p)$-balance condition}, if there is a constant $C>0$ such that
\begin{equation}\label{e.qp_balance}
\frac{\diam(B')}{\diam(B)} \biggl(\frac{\nu(B')}{\nu(B)}\biggr)^{1/q}\le
C^{1/p}\biggl(\frac{\mu(B')}{\mu(B)}\biggr)^{1/p} 
\end{equation}
whenever $B$ and $B'=B(x',r')$ are balls in $X$ such that $x'\in B$ and $0<r'\le \diam(B)$.
\end{definition}

We call the $(p,p)$-balance condition simply {\em $p$-balance condition}; in this 
case inequality~\eqref{e.qp_balance} is more conveniently written as
\begin{equation}\label{e.p_balance}
\biggl(\frac{\diam(B')}{\diam(B)}\biggr)^{p} \frac{\nu(B')}{\nu(B)}\le C\frac{\mu(B')}{\mu(B)}\,,
\end{equation}
where  $B$ and $B'$ are as in Definition~\ref{d.qp_balance}.

In the following example we consider the  special  case $\nu=\mu$. 
\begin{example}\label{e.q=p}
If $X=\R^n$ is equipped with the standard Euclidean metric and two copies of the
Lebesgue measure $\nu=\mathcal{L}_n=\mu$, then 
$(X,d,\nu,\mu)$ satisfies
a $(q,p)$-balance condition with $q=np/(n-p)>p$ for all $1\le p<n$.
More generally, let $1\le p<\infty$ and let $X=(X,d,\nu,\mu)$ be a metric
two-measure space such that $\nu=\mu$. Then, by inequality \eqref{e.radius_measure}, 
there exists  $q>p$ such that
$X$ satisfies a $(q,p)$-balance condition.
Furthermore, if $1<p<\infty$, then by Proposition \ref{p.basic_balance}(D) below 
we find that $X$ satisfies a $(p,p-\eps)$-balance condition for some $\eps>0$.
This fact explains why the balance condition does not play a visible
role in the Keith--Zhong self-improvement results for one measure 
Poincar\'e inequalities;  cf.~\cite{MR2415381, SEB, KinnunenLehrbackVahakangasZhong}.
\end{example}

Next we establish some basic relations between different  
balance conditions. 
Below, the statement {\em $(q,p)$-balance condition {\rm(}or $p$-balance condition, resp.{\rm)}}\ means
that $X$ satisfies the respective balance condition.

\begin{proposition}\label{p.basic_balance}
Let $1\le p,q<\infty$, and assume that $X=(X,d,\nu,\mu)$ is a connected
metric two-measure space. Then the following statements hold: 
\begin{itemize}
\item[(A)] A $(q,p)$-balance condition implies $(\lambda q,\lambda p)$-balance  conditions  for every  $\lambda \geq 1$. 
\item[(B)] A $p$-balance condition implies $q$-balance conditions for all $1\le p\le q<\infty$.
\item[(C)] A $(q,p)$-balance condition implies $(q',p')$-balance conditions for all $1\le q'\leq q$ and $p'\geq p$. 
\item[(D)] If $1<p,q<\infty$ and $X$ satisfies a $(q,p)$-balance condition, then  
for every $0<\delta\le q-1$ there is $0<\varepsilon\le p-1$ such that $X$ satisfies a $(q-\delta,p-\varepsilon)$-balance condition.
\end{itemize}
\end{proposition}
\begin{proof}
During the proof of the proposition, we  assume  that $B$ and $B'=B(x',r')$ are balls in $X$ such that $x'\in B$ and $0<r'\le \diam(B)$.

(A) Taking inequality~\eqref{e.qp_balance} to power $1/\lambda$ yields 
\[
\left(\frac{\diam(B')}{\diam(B)}\right)^{1/\lambda} \biggl(\frac{\nu(B')}{\nu(B)}\biggr)^{1/(\lambda q)}\le
C^{1/(p\lambda)}\biggl(\frac{\mu(B')}{\mu(B)}\biggr)^{1/(\lambda p)}\,, 
\]
where $C>0$ is the constant in the $(q,p)$-balance condition. Now the claim follows from the fact that $\diam(B')/\diam(B)\leq 2$.

(B) This follows from statement (A) by taking $\lambda=q/p\ge 1$. 

(C) This follows from the facts that $B'\subset 4B$ and that both $\mu$ and $\nu$ are doubling measures, and thus satisfy inequalities \eqref{e.doubling}.

(D) Let  $s'>0$ and $\sigma'>0$ 
be the exponents as in \eqref{e.radius_measure} and \eqref{e.rev_d} 
for  $\mu$ and $\nu$, respectively. We remark that
such a $\sigma'>0$ exists, since $X$ is connected.
Then we have, as in~\eqref{e.nu_mu_comparison}, 
that
\[
\frac{\nu(B')}{\nu(B)}\leq 
C\left(\frac{\mu(B')}{\mu(B)}\right)^{\sigma'/s'}\,,
\]
where the constants are independent of $B'$ and $B$.
Fix $0<\delta\le q-1$ and then choose $0<\varepsilon\le p-1$ such that
\[
 \frac{1}{p-\varepsilon}-\frac{1}{p}\leq  \left( \frac{1}{q-\delta}-\frac{1}{q} \right)\frac{\sigma'}{s'}\,.
\]
By the assumed $(q,p)$-balance condition 
and the fact that $\mu(B')\le c_\mu^2\mu(B)$, 
we  thus  obtain
\[
\begin{split}
\frac{\diam(B')}{\diam(B)} \biggl(\frac{\nu(B')}{\nu(B)}\biggr)^{1/(q-\delta)} &=\frac{\diam(B')}{\diam(B)} \biggl(\frac{\nu(B')}{\nu(B)}\biggr)^{1/q} \biggl(\frac{\nu(B')}{\nu(B)}\biggr)^{1/(q-\delta)-1/q}\\
&\le C\biggl(\frac{\mu(B')}{\mu(B)}\biggr)^{1/p} \biggl(\frac{\mu(B')}{\mu(B)}\biggr)^{(\sigma'/s')(1/(q-\delta)-1/q)}\\
&\leq C\biggl(\frac{\mu(B')}{\mu(B)}\biggr)^{1/p} \biggl(\frac{\mu(B')}{\mu(B)}\biggr)^{1/(p-\eps)-1/p} \\
&=  C \biggl(\frac{\mu(B')}{\mu(B)}\biggr)^{1/(p-\eps)}\,.
\end{split}
\]
The desired $(q-\delta,p-\eps)$-balance condition follows. \end{proof}

In some of our self-improvement results we need to assume {\em a priori} that $X$ satisfies a slightly better balance condition
than a $p$-balance condition. To this end, we define
the notion of a {\em bumped} $p$-balance condition as follows.

\begin{definition}\label{d.bumped_balance_def}
Let $\Psi\colon (0,\infty)\to (0,\infty)$ be a function.
We say that a metric (or a geodesic) two-measure space
$(X,d,\nu,\mu)$ 
{\em satisfies a $\Psi$-bumped $p$-balance condition},  if
\begin{equation}
\biggl(\frac{\diam(B')}{\diam(B)}\biggr)^{p} \frac{\nu(B')}{\nu(B)}\le
\Psi\biggl(\frac{\diam(B')}{\diam(B)}\biggr)\frac{\mu(B')}{\mu(B)} 
\end{equation}
whenever $B$ and $B'=B(x',r')$ are balls in $X$ such that $x'\in B$ and $0<r'\le \diam(B)$.
\end{definition}

The following result shows that under mild assumptions,
a $\Psi$-bumped $p$-balance condition is equivalent to  a $(p-\varepsilon)$-balance condition
for some $\eps>0$.

\begin{theorem}\label{t.bump_impro}
Let $1< p<\infty$, and assume 
that $X=(X,d,\nu,\mu)$ is a connected metric two-measure space.
Then the following conditions are equivalent:
\begin{enumerate}
\item[(A)] $X$ satisfies  a  $(p-\varepsilon)$-balance condition for some $0<\varepsilon< p-1$.
\item[(B)] $X$ satisfies  a  $\Psi$-bumped $p$-balance condition with a function $\Psi\colon (0,\infty)\to (0,\infty)$ for which there exists $t_{0}>0$ and $0<\delta<1$ 
such that $\Psi(t)\leq \delta$ for all $0<t\leq t_{0}$.
\end{enumerate}
\end{theorem}

\begin{proof} 
The implication (A) $\Longrightarrow$ (B) follows by choosing $\Psi(t)=Ct^\eps$ for each $t>0$, where
$C>0$ is the constant  in  the $(p-\eps)$-balance condition~\eqref{e.p_balance}.

(B) $\Longrightarrow$ (A).
Without loss of generality, we may assume that (B) holds for some $0<t_0<1$. 
Let $B=B(x,r)$ and $B'=B(x',r')$ be balls in $X$ such that $x'\in B$ and $0<r'\le \diam(B)$.
By the doubling property \eqref{e.doubling} of the measures  and the assumption that $X$ is connected, we may in addition assume that  $\diam(B')\le (t_{0}/2)\diam(B)$. 
Let $j\in \mathbb Z$ and $0<t\le t_{0}/2$ be such that $(t_{0}/2)^{j+1}\diam(B)<\diam(B')\le(t_{0}/2)^j\diam(B)$ and 
$\diam(B')=t^j \diam(B)$. It follows that $j\geq 1$ and $(t_{0}/2)^2< t\leq t_{0}/2<1/2$. 
Let \[B_{0}=B(x_{0},r_{0})=B(x',r')=B'\,,\] choose $r_{i}=t^{-i}r_{0}$ and $B_{i}=B(x',r_{i})$ for $1\leq i<j$, and
finally let $B_{j}=B$.  Fix $0\le i <j$.  Since $\diam(B')<\diam(B)\le\diam(X)$, it holds that 
\[r_i=t^{-i}r_{0}\le t^{-j+1}r_{0}< 2^{-1}t^{-j}\diam(B')=2^{-1}\diam(B)\,,\]
and so $r_i\le \diam(B_i)\le 2r_i$;   
here we have also used the assumption that $X$ is connected.
From the previous estimates it follows that  $\diam(B_{i-1})/\diam(B_{i}) \le 2r_{i-1}/r_i = 2t\le t_0$ for all $1\le i <j$,  and also that
\[\diam(B_{j-1})/\diam(B_{j})\le 2t^{-j+1}\diam(B')/\diam(B)=2t\le t_0.\]

Hence, by applying condition (B), we obtain for each $1\le i \le j$ that
\[
\biggl(\frac{\diam(B_{i-1})}{\diam(B_{i})}\biggr)^{p} \frac{\nu(B_{i-1})}{\nu(B_{i})}\le
\Psi\biggl(\frac{\diam(B_{i-1})}{\diam(B_{i})}\biggr)\frac{\mu(B_{i-1})}{\mu(B_{i})} 
\le \delta \frac{\mu(B_{i-1})}{\mu(B_{i})}\,.
\]
By multiplying these inequalities we thus obtain
\begin{equation}\label{eq.product}
\frac{\diam(B')^p}{\diam(B)^p}\frac{\nu(B')}{\nu(B)}= 
\prod_{i=1}^{j}\frac{\diam(B_{i-1})^p}{\diam(B_{i})^p}\frac{\nu(B_{i-1})}{\nu(B_{i})}\leq 
\prod_{i=1}^{j}\delta\frac{\mu(B_{i-1})}{\mu(B_{i})}= \delta^j\frac{\mu(B')}{\mu(B)}\,.
\end{equation}
Now choose 
\[
\eps = \frac{\log \delta}{\log((t_0/2)^2)} >0.
\] 
Then $\log( (t_0/2)^{2\eps}) = \log \delta$, 
and so 
$\delta^j=(t_0/2)^{2j\eps}\le t^{j\eps}$.
Thus we conclude from~\eqref{eq.product} that
\[
\biggl(\frac{\diam(B')}{\diam(B)}\biggr)^p\frac{\nu(B')}{\nu(B)}
\le \delta^j\frac{\mu(B')}{\mu(B)} 
\le t^{j\eps}\frac{\mu(B')}{\mu(B)}
= \biggl(\frac{\diam(B')}{\diam(B)}\biggr)^\eps\frac{\mu(B')}{\mu(B)}\,.
\]
In the case $\eps<p-1$,  we see that
$X$ satisfies 
the $(p-\eps)$-balance condition with $0<\eps<p-1$.
If $\eps \ge p-1$, we replace $\eps$ with $(p-1)/2$, and  the claim follows. 
\end{proof}

The following  example  shows 
that a mere $p$-balance condition does 
not  imply a $(p-\eps)$-balance condition.
We will return to this example later in connection
with Poincar\'e inequalities;  cf.\ Example~\ref{ex.no_impro}.

\begin{example}\label{ex.bal_no_impro}
Consider $X=\R^n$ equipped with the standard Euclidean metric  $d$  and let $\mu$ be the $n$-dimensional Lebesgue measure  $\mathcal L^n$
on $\R^n$.  Fix $1< p< n$ and let $w(x)=\lvert x\rvert^{-p}$ if $x\in\R^n\setminus\{0\}$. Then the weight
$w$ belongs to the Muckenhoupt class $A_1(\mu)\subset A_\infty(\mu)$; see e.g.~\cite[p.~229]{Torchinsky1986}. 
We let $\nu$ be the $w$-weighted Lebesgue measure, that is, 
\[d\nu(x)=w(x)\,d\mu(x)=\lvert x\rvert^{-p}\,d\mu(x)\,.\]
Then $X=(X,d,\nu,\mu)$ is a geodesic two-measure
space 
that satisfies the $p$-balance condition.
Indeed,  the $p$-balance condition  can be established by considering the cases $0\in 8B$ and $0\not\in 8B$
separately and applying the $A_1(\mu)$-property
of $w$ in the former case; here the ball  $B\subset \R^n$ is as in Definition \ref{d.qp_balance} with $q=p$.

On the other hand, 
by the $A_1(\mu)$-property of $w$, 
there is a constant $c>1$ such that
\[c^{-1} r^{n-p}\le \nu(B(0,r))\le cr^{n-p}\] for all $r>0$. Hence, if $0<\eps\le p-1$
and $0<r'<r$, it holds for balls
$B=B(0,r)$ and $B'=B(0,r')$ that
\[
\biggl(\frac{\diam(B')}{\diam(B)}\biggr)^{p-\eps} \frac{\nu(B')}{\nu(B)} 
\ge C \biggl(\frac{r'}{r}\biggr)^{p-\eps} \biggl(\frac{r'}{r}\biggr)^{n-p}
\ge C \biggl(\frac{r'}{r}\biggr)^{-\eps} \frac{\mu(B')}{\mu(B)}
\]
with constants independent of both $B$ and $B'$. 
Keeping $r$ fixed and letting $r' \to 0$ shows that $X=(X,d,\nu,\mu)$ does not
satisfy a $(p-\eps)$-balance condition for any $0<\eps\le p-1$. 
\end{example}

\section{Poincar\'e inequalities}\label{s.two_measure_PI}

Let  $1\le p<\infty$ and let  $X=(X,d,\nu,\mu)$ be a metric two-measure space; recall  Section~\ref{s.metric}. 
We say that a $\mu$-measurable function $g\colon X\to [0,\infty]$ is a {\em $p$-weak upper gradient (w.r.t.\ $X$)}
of a function $u\colon X\to \R$ 
if inequality
\begin{equation}\label{e.modulus}
\lvert u(\gamma(0))-u(\gamma(\ell_\gamma))\rvert \le \int_\gamma g\,ds
\end{equation}
holds for $p$-almost every curve $\gamma\colon [0,\ell_\gamma]\to X$; that is, there exists a non-negative  Borel function 
$\rho\in L^p_{\mathrm{loc}}(X;d\mu)$  such that 
$\int_\gamma \rho\,ds=\infty$ whenever inequality~\eqref{e.modulus} does not hold or is not defined.
We refer to~\cite{MR2867756,MR1800917,MR3363168} for further information
on $p$-weak upper gradients.

\begin{definition}
Let $1< p<\infty$. 
For a Lipschitz function $u\in \Lip(X)$, we let $\mathcal{D}^{p}(u)=\mathcal{D}^{p}(u;d\mu)$ be the set of
all  $p$-weak upper gradients $g\in L^p_{\textup{loc}}(X;d\mu)$ of $u$.
\end{definition}

The following
conditions (D1)--(D3) hold for all Lipschitz functions $u,v\colon X\to \R$:
\begin{itemize}
\item[(D1)] $\lvert a\rvert g\in \mathcal{D}^{p}(au)$ if $a\in\R$ and $g\in \mathcal{D}^{p}(u)$,
\item[(D2)] $g + \hat g\in \mathcal{D}^{p}(u+v)$  if $g\in\mathcal{D}^{p}(u)$ and $\hat g\in\mathcal{D}^{p}(v)$, 
\item[(D3)] If $v\colon X\to \R$ is $\kappa$-Lipschitz function with a constant $\kappa\ge 0$,  
$v|_{X\setminus E}=u|_{X\setminus E}$ for a Borel set $E\subset X$,  and $g\in \mathcal{D}^{p}(u)$,  then
$\kappa \mathbf{1}_{E} + g \mathbf{1}_{X\setminus E}\in\mathcal{D}^{p}(v)$.
\end{itemize}
The properties (D1) and (D2) are rather well known, see for instance \cite[Corollary 1.39]{MR2867756}. 
The property (D3) is a consequence of  the  `Glueing lemma', see e.g.~\cite[Lemma 2.19, Remark 2.28]{MR2867756}.

The family $\mathcal{D}^{p}(u)$ has  the following
{\em minimality property}:
if  $u\in \Lip(X)$, then there exists
a $p$-weak upper gradient
$g_u\in \mathcal{D}^{p}(u)$ such that $g_u\le g$ pointwise $\mu$-almost
everywhere if $g\in \mathcal{D}^{p}(u)$; we refer to~\cite[Theorem 2.25]{MR2867756}.
This function $g_u$ is called the {\em minimal $p$-weak upper gradient of $u$,} and it is unique up to sets of $\mu$-measure zero in $X$. 

Let  $1< q\le p<\infty$.  If $u\in \Lip(X)$, then it is clear that $\mathcal{D}^p(u)\subset \mathcal{D}^{q}(u)$.
However, in general, it is not true that $\mathcal{D}^p(u)\supset \mathcal{D}^{q}(u)$;
see~\cite{MR3411142}. Therefore, in some of our results we need to explicitly assume that the following
indepdence property is valid for a suitable exponent.

\begin{definition}\label{d.independence}
Let $1<p<\infty$. A metric two-measure space $X$ has the {\em $p$-independence property}, if
 for all $q\ge p$ and for all $u\in \Lip(X)$ 
the minimal $q$-weak upper gradient $h_u\in\mathcal{D}^q(u)$ of $u$ coincides $\mu$-almost everywhere
with the minimal $p$-weak upper gradient $g_{u}\in\mathcal{D}^p(u)$.
\end{definition}

In other words,  the  $p$-independence property
means that the minimal $q$-weak upper gradient of $u\in \Lip(X)$
is independent of $q\ge p$. We  emphasize that this property does not depend
on the measure $\nu$ at all.

Let us provide some examples when a $p$-independence property holds.

\begin{example}\label{e.p_indep}  Let $1<p<\infty$.
If $X=(X,d,\nu,\mu)$ is a complete metric two-measure space and  the space  $Y=(X,d,\mu,\mu)$
supports  a  $(1,p)$-Poincar\'e inequality 
in the sense of the following Definition~\ref{d.poincare}, then
$X$ has the $p$-independence property. This 
follows from~\cite[Corollary A.8, Theorem 4.15]{MR2867756} and the fact that $\mu$ is
doubling.
\end{example}

\begin{definition}\label{d.poincare}
Let $1\le q,p<\infty$.
We say that a  metric two-measure space $X=(X,d,\nu,\mu)$ supports a 
{\em $(q,p)$-Poincar\'e inequality}, if
there exists a constant $K_{q,p}>0$ such that 
\[
\biggl(\vint_{B} \lvert u(x)-u_{B;\nu}\rvert^q\,d\nu(x)\biggr)^{1/q}
\le K_{q,p}^{1/p}\diam(B)\biggl(\vint_{B} g(x)^p \,d\mu(x)\biggr)^{1/p}
\]
whenever $B$ is a ball in $X$, and  $u\in\Lip(X)$ and $g\in \mathcal{D}^{p}(u)$.
\end{definition}

If a connected metric two-measure space $X$ supports a $(q,p)$-Poincar\'e inequality, then $X$ also satisfies a $(q,p)$-balance condition. This
follows from Lemma \ref{l.Poincare_implies_balance}  below.
The proof of this lemma follows the argument given in~\cite{ChanilloWheeden1985} 
for the special case $X=\R^n$.
 Note that the  integration on  the right-hand side
of inequality~\eqref{e.assumed} is taken over the whole space  $X$.


\begin{lemma}\label{l.Poincare_implies_balance}
Let $1\le q,p<\infty$. Assume that
$X=(X,d,\nu,\mu)$ is a connected metric two-measure space and 
that there is a constant $C_1>0$ such that inequality
\begin{equation}\label{e.assumed}
\biggl(\vint_{B} \lvert u(y)-u_{B;\nu}\rvert^q\,d\nu(y)\biggr)^{1/q}\le C_1\diam(B)\biggl(\frac{1}{\mu(B)}\int_X g(y)^{p}\,d\mu(y)\biggr)^{1/p}
\end{equation}
holds whenever $B$ is a ball in $X$, $u\in \Lip(X)$ has a bounded support, and $g\in \mathcal{D}^p(u)$.
Then $X$ satisfies a $(q,p)$-balance condition.
\end{lemma}

\begin{proof}
We  need to show  that inequality~\eqref{e.qp_balance} holds for all balls $B$ and $B'=B(x',r')$ in $X$
such that $x'\in B$ and $0<r'\le \diam(B)$.
By using inequalities \eqref{e.doubling} and \eqref{e.diams}, and~\cite[Lemma 3.7]{MR2867756}, 
we can furthermore assume that
\[
2B'\subset B\,,\qquad \frac{\nu(2B')}{\nu(B)}\le \frac{1}{8}\,,\quad\text{ and }\quad \frac{\nu(B')}{\nu(B'\setminus 2^{-1}B')}\le C(c_\nu)
\]
for some constant $C(c_\nu)>0$  only  depending on $c_\nu$.  (This  reduction is straightforward but 
slightly tedious to establish, and hence  we leave  the  details to the interested reader.)
Let $\varphi\colon X\to \R$ be the $1/r'$-Lipschitz function that is defined  for all $y\in X$  by
\[
\varphi(y)=\max\biggl\{1-\frac{d(y,B')}{r'},0\biggr\}\,.
\]
Observe that $0\le \varphi\le 1$, that $\varphi=1$ in $B'$, and that $\varphi=0$ in $X\setminus 2B'$.

By condition (D3), applied with $E=2B'$, we see that
$(1/r')\mathbf{1}_{2B'}\in \mathcal{D}^p(\varphi)$.
Define a Lipschitz function $u\colon X\to [0,\infty)$ with  bounded  support 
by setting $u(y)=d(y,x')\varphi(y)$ for every $y\in X$. By the product rule \cite[Theorem 2.15]{MR2867756}, 
\[
g=\frac{d(\cdot,x')}{r'}\mathbf{1}_{2B'} + \varphi \in \mathcal{D}^p(u)\,.
\]
Notice, in particular, that $g\le 3\cdot \mathbf{1}_{2B'}$.
We estimate
\begin{align*}
u_{B;\nu}=\frac{1}{\nu(B)}\int_{B} d(y,x')\varphi(y)\,d\nu(y)
\le \frac{1}{\nu(B)}\int_{2B'} d(y,x')\,d\nu(y)\le 2r'\frac{\nu(2B')}{\nu(B)}\le \frac{r'}{4}\,.
\end{align*}
Now, for every $y\in B'\setminus 2^{-1}B'$,  we have
$u(y)= d(y,x')\varphi(y)\ge r'/2$.
As a consequence, we obtain 
\begin{align*}
\frac{1}{\nu(B')}\int_B\lvert u(y)-u_{B;\nu}\rvert^q\,d\nu(y)&\ge \vint_{B'} \lvert u(y)-u_{B;\nu}\rvert^q\,d\nu(y)\\
&\ge C(c_\nu)^{-1}\vint_{B'\setminus 2^{-1}B'} \lvert u(y)-u_{B;\nu}\rvert^q\,d\nu(y)\ge C(c_\nu)^{-1}\biggl(\frac{r'}{4}\biggr)^q\,.
\end{align*}
On the other hand, by the assumed inequality \eqref{e.assumed},
\begin{align*}
\biggl(&\frac{1}{\diam(B)^q}\vint_{B} \lvert u(y)-u_{B;\nu}\rvert^q\,d\nu(y)\biggr)^{p/q}\le C_1^p\frac{1}{\mu(B)}\int_X g(y)^{p}\,d\mu(y)\le C(C_1,p)\frac{\mu(2B')}{\mu(B)}\,.
\end{align*}

By combining the estimates above, we see that 
\begin{align*}
 C(c_\nu)^{-1}\biggl(\frac{r'}{4\diam(B)}\biggr)^q\frac{\nu(B')}{\nu(B)}
&\le \frac{1}{\diam(B)^q}\vint_{B} \lvert u(y)-u_{B;\nu}\rvert^q\,d\nu(y) \\
&\le \biggl(C(C_1,p)\frac{\mu(2B')}{\mu(B)}\biggr)^{q/p}\,.
\end{align*} 
That is, 
\[
\frac{\diam(B')}{\diam(B)}\biggl(\frac{\nu(B')}{\nu(B)}\biggr)^{1/q}
 \le C(q,p,c_\nu,c_\mu,C_1)
 \biggl(\frac{\mu(B')}{\mu(B)}\biggr)^{1/p}\,,
\]
which is the desired inequality \eqref{e.qp_balance}
with a constant $C=C(q,p,c_\nu,c_\mu,C_1)>0$
independent of the balls $B$ and $B'$.
\end{proof}

The following example illustrates a case where a $(p,p)$-Poincar\'e inequality does not
improve to a $(p,p-\eps)$-Poincar\'e inequality --- 
or even to a $(p-\eps,p-\eps)$-Poincar\'e inequality ---  for any  $0<\eps\le p-1$.
The obstruction is that there is no $(p-\eps)$-balance condition.
This obstruction cannot appear
 in  a metric space $X=(X,d,\mu,\mu)$  equipped with a single measure $\mu$; 
see Example \ref{e.q=p}. Therefore, in such a geodesic 
space $X$, a $(p,p)$-Poincar\'e inequality 
always implies a $(p,p-\eps)$-Poincar\'e inequality for some $\eps>0$;
cf.~\cite{DeJarnette,SEB,MR2415381,KinnunenLehrbackVahakangasZhong}.

\begin{example}\label{ex.no_impro}
Let $1<p<n$.
Let $X=(X,d,\nu,\mu)$ be the geodesic two-measure space as in Example \ref{ex.bal_no_impro}.
It follows from the results in~\cite{ChanilloWheeden1985CPDE} that $X$  supports a 
$(p,p)$-Poincar\'e inequality; see also \cite[Remark~1.6]{FranchiLuWheeden1995}
and \cite[Corollary 3.2]{MR1609261}.
By Example \ref{ex.bal_no_impro},
we find that $X$ does not
satisfy a $(p-\eps)$-balance condition for any $0<\eps\le p-1$. Lemma \ref{l.Poincare_implies_balance} then implies that $X$ 
does not support a $(p-\eps,p-\eps)$-Poincar\'e inequality for any $0<\eps\le p-1$. Moreover, by H\"older's
inequality, we find that $X$ does not support
a $(p,p-\eps)$-Poincar\'e inequality for
any $0<\eps\le p-1$.
\end{example}

\section{Maximal Poincar\'e inequalities}\label{s.max_PI}

In the light of Example \ref{ex.no_impro}, it  is  
clear that
Poincar\'e inequalities in metric two-measure spaces are not always self-improving.
 In this section we introduce slightly stronger variants of Poincar\'e inequalities, which
turn out to be 
more amenable to self-improvement.  These maximal Poincar\'e inequalities are defined
in terms of the following sharp maximal functions.

\begin{definition}\label{d.m_def}
Let $1\le q<\infty$.
If $\mathcal{B}\not=\emptyset$ is a family
of balls in a metric two-measure space $X=(X,d,\nu,\mu)$ and $u\colon X\to \R$ is a Lipschitz function, then we define a sharp maximal function
\[
M^{\nu,q}_{\mathcal{B}}u(x)=\sup_{B\colon x\in B\in \mathcal{B}} \biggl(\frac{1}{\diam(B)^q}\vint_B \lvert u(y)-u_{B;\nu}\rvert^q\,d\nu(y)\biggr)^{1/q}\,,\qquad x\in X\,.
\]
The supremum above is defined to be zero, if there is no ball $B$ in $\mathcal{B}$ such that $x\in B$. 
\end{definition}

We remark that  if $u\colon X\to \R$ is a $\kappa$-Lipschitz function, then 
$M^{\nu,q}_{\mathcal{B}}u(x)\le \kappa$ for every $x\in X$. This  fact follows from the estimate
\[
\lvert u(y)-u_{B;\nu}\rvert^q
\le \vint_{B} \lvert u(y)-u(z)\rvert^q\,d\nu(z)\le \kappa^q\,\diam(B)^q\,,\qquad y\in B\in\mathcal{B}\,.
\]

When $B_0\subset X$ is a fixed ball, we will often use the sharp maximal function $M^{\nu,q}_{\mathcal{B}_0}u(x)$ with
\begin{equation}\label{eq.ball_family_0}
\mathcal{B}_0=\{B(x,r)\,:\,B(x,2r)\subset B_0\}\,,
\end{equation}
which we call the \emph{family of balls associated with} $B_0$.
This family is needed, for instance, in the following definition
of the  so-called maximal Poincar\'e inequalities. 

\begin{definition} Let $1\le q,p<\infty$.
We say that a metric two-measure space $X=(X,d,\nu,\mu)$ supports a {\em maximal $(q,p)$-Poincar\'e inequality}, with a constant  $C>0$,
if for every ball $B_0\subset X$ it holds that
\begin{equation*}
\begin{split}
\int_{B_0}\bigl( M^{\nu,q}_{\mathcal{B}_0} u(x)\bigr)^{p}\,d\mu(x)
\le C\int_{B_0} g(x)^p\,d\mu(x)
\end{split}
\end{equation*}
whenever
$u\in\Lip(X)$ and $g\in\mathcal{D}^{p}(u)$,
where $\mathcal{B}_0$ is the family~\eqref{eq.ball_family_0} of balls associated with $B_0$.
\end{definition}

The following lemma yields basic  
relations between Poincar\'e type inequalities and their maximal analogues.

\begin{lemma}\label{l.straightforward}
Let $1<p,q<\infty$ and assume that $X=(X,d,\nu,\mu)$ 
is  a metric two-measure space.
If $X$ supports a 
maximal $(q,p)$-Poincar\'e inequality, then there is a constant $C>0$ such that \begin{equation}\label{e.weak_poincare}
\biggl(\vint_{B} \lvert u(x)-u_{B;\nu}\rvert^q\,d\nu(x)\biggr)^{1/q}
\le C\diam(B)\biggl(\vint_{2B} g(x)^p \,d\mu(x)\biggr)^{1/p}
\end{equation}
whenever $B$ is a ball in $X$ and whenever 
$u\in\Lip(X)$ and $g\in \mathcal{D}^{p}(u)$.

Conversely, if $X$ supports a $(q,p-\eps)$-Poincar\'e inequality,
for some $0<\eps<p-1$, then  $X$ supports a maximal $(q,p)$-Poincar\'e inequality.
\end{lemma}

\begin{proof}
Assume first 
that $X$ supports a maximal $(q,p)$-Poincar\'e inequality
with a constant $c>0$. 
Fix a ball $B\subset X$, denote $B_0=2B$, and
let $\mathcal{B}_0=\{B(x,r)\,:\,B(x,2r)\subset B_0\}$
be the ball family associated with $B_0$. 
Since $B\in\mathcal{B}_0$, we have
for all $u\in\Lip(X)$ and all $g\in\mathcal{D}^p(u)$ that
\begin{align*}
&\mu(B)\biggl(\frac{1}{\diam(B)^q}\vint_B \lvert u(y)-u_{B;\nu}\rvert^q\,d\nu(y)\biggr)^{p/q}
\le \int_{B} \bigl(M^{\nu,q}_{\mathcal{B}_0} u(x)\bigr)^{p}\,d\mu(x)\\
&\qquad\le \int_{B_0} \bigl(M^{\nu,q}_{\mathcal{B}_0} u(x)\bigr)^{p}\,d\mu(x)
\le c\int_{B_0} g(x)^p \,d\mu(x)\,.
\end{align*}
Inequality \eqref{e.weak_poincare} follows  from the above estimate,  the doubling
property \eqref{e.doubling} for $\mu$, and the fact that $B_0=2B$.

For the converse implication, we assume that
$X$ supports a $(q,p-\eps)$-Poincar\'e inequality
for some $0<\eps<p-1$ and with a constant $K_{q,p-\eps}>0$. 
Fix a ball $B_0\subset X$ and let 
$\mathcal{B}_0$ be again the associated family of balls, as above. 
Fix $u\in\Lip(X)$ and $g\in\mathcal{D}^{p}(u)\subset \mathcal{D}^{p-\eps}(u)$, and 
let $x\in X$ 
 be such  that $x\in B\in\mathcal{B}_0$.
Then, by the assumed
$(q,p-\eps)$-Poincar\'e inequality, 
\begin{align*}
\biggl(\frac{1}{\diam(B)^q}\vint_B \lvert u(y)-u_{B;\nu}\rvert^q\,d\nu(y)\biggr)^{1/q}
&\le K_{q,p-\eps}^{1/(p-\eps)}\biggl(\vint_{B} g(y)^{p-\eps} \,d\mu(y)\biggr)^{1/(p-\eps)}\\
&\le  K_{q,p-\eps}^{1/(p-\eps)}\bigl(M^\mu(\mathbf{1}_{B_0} g^{p-\eps})(x)\bigr)^{1/(p-\eps)}\,.
\end{align*}
Here $M^\mu$ is the noncentered Hardy--Littlewood maximal 
operator
with respect to measure $\mu$; see e.g.\ \cite[Section 3.2]{MR2867756}. 
Hence, by using the definition of $M_{\mathcal{B}_0}^{\nu,q}u$, we see that
\[
\bigl(M_{\mathcal{B}_0}^{\nu,q}u(x)\bigr)^p\le K_{q,p-\eps}^{p/(p-\eps)}\bigl(M^\mu(\mathbf{1}_{B_0} g^{p-\eps})(x)\bigr)^{p/(p-\eps)}\qquad 
\text{ for all } x\in X\,.
\]
By the boundedness of  the  Hardy--Littlewood maximal operator in $L^{p/(p-\eps)}(X;d\mu)$, we find that 
\begin{align*}
\int_{B_0} \bigl(M_{\mathcal{B}_0}^{\nu,q}u(x)\bigr)^p d\mu(x)
&\le K_{q,p-\eps}^{p/(p-\eps)}\int_X \bigl(M^\mu(\mathbf{1}_{B_0} g^{p-\eps})(x)\bigr)^{p/(p-\eps)}\,d\mu(x)\\
&\le C(c_\mu,p,\eps)K_{q,p-\eps}^{p/(p-\eps)}\int_{B_0} g(x)^{p}\,d\mu(x)\,.
\end{align*}
From this estimate it follows that $X$ supports a maximal $(q,p)$-Poincar\'e inequality.
\end{proof}

\begin{remark}
Lemma \ref{l.straightforward} is sharp in the sense that
a $(p,p)$-Poincar\'e inequality does not always
imply a maximal $(p,p)$-Poincar\'e
inequality;  Example~\ref{ex.sharp_explanation} below  provides
a concrete example of such a situation. We will use our main results
to provide this example, and therefore the construction is postponed 
 to the end of  Section~\ref{s.mainR}.
\end{remark}

\section{Main results}\label{s.mainR}

This section contains the main results of this work. 
First, in Section \ref{s.si_I}, we provide sufficient
conditions for the self-improvement of $(p,p)$-Poincar\'e inequalities. The self-improvement  properties
of maximal Poincar\'e inequalities and their variants are  presented in Section~\ref{s.si_II}.
The main lines of the proofs are also given in this section, but here we rely on some technical results
whose statements and proofs are postponed to the final sections of this work. 

Recall that we have as a standing assumption that $\mu$ and $\nu$
are doubling Borel measures in $X$; cf.\ Section~\ref{s.metric}. 
All other assumptions concerning the space $X=(X,d,\nu,\mu)$
are stated separately in each of the following results.

\subsection{Results for Poincar\'e inequalities}\label{s.si_I}

Under  certain assumptions, 
a $(p,p)$-Poincar\'e inequality and a $(p-\tau)$-balance condition,  for some  $0<\tau<p-1$, 
together imply a maximal $(p,p-\eps)$-Poincar\'e inequality for some $\eps>0$. 
 By Lemma \ref{l.straightforward}, this  can be viewed
as a self-improvement  result  for Poincar\'e inequalities in geodesic two-measure spaces.

We begin with Theorem \ref{t.main_thm} below. It is a consequence
of Lemma \ref{l.Lip_estimate} and Theorem \ref{t.main_local}, whose rather technical  formulation and lengthy proof
are given in 
Section~\ref{s.main}.

\begin{theorem}\label{t.main_thm}
Let $1<p<\infty$ and $0<\tau,\vartheta<p-1$,  and assume 
that $X=(X,d,\nu,\mu)$  is a geodesic two-measure space satisfying the following
assumptions: 
\begin{itemize}
\item
$X$ supports a $(p,p)$-Poincar\'e inequality with 
a constant $K_{p,p}>0$ (Definition \ref{d.poincare})
\item $X$ satisfies  a $(p-\tau)$-balance condition
with a constant $C_b>0$ (Definition \ref{d.qp_balance})
\item $X$ has the $(p-\vartheta)$-independence property (Definition \ref{d.independence})
\item 
 $\nu$ is an $A_\infty(\mu)$-weighted measure, 
 with constants  $c_{\nu,\mu},\delta>0$ (Definition \ref{d.comparability}).
\end{itemize} 
Then there exists $0<\eps_0<p-1$ such that $X$ supports a maximal $(p,p-\eps)$-Poincar\'e inequality for every $0\le \eps\le \eps_0$.
\end{theorem}
 
 For the proof of Theorem \ref{t.main_thm}, we need the following lemma.
 
\begin{lemma}\label{l.Lip_estimate}

Let $1<p<\infty$ and assume that  $X=(X,d,\nu,\mu)$  is a geodesic two-measure space. 
Let $B_0\subset X$ be a ball and let $\mathcal{B}_0$ 
be the family~\eqref{eq.ball_family_0} of balls associated with $B_0$. 
If $u\colon X\to \R$ is a Lipschitz function and 
$g_u\in L^p_{\textup{loc}}(X;d\mu)$ is the minimal $p$-weak upper gradient of $u$, 
then inequality
\begin{equation}\label{e.desired}
g_u(x)\le C(c_\nu) M^{\nu,p}_{\mathcal{B}_0}u(x)
\end{equation}
holds for $\mu$-almost every $x\in B_0$.
\end{lemma}

\begin{proof}
The proof of \cite[Lemma 5.3]{KinnunenLehrbackVahakangasZhong} can be easily
adapted to the present setting of two measures, yielding a proof of the claim; the main difference is that here $\nu$ is used
to define $M^{\nu,p}_{\mathcal{B}_0}u$, whereas
$\mu$ is used to define $p$-weak
upper gradients.  That is, in \cite[Lemma 5.3]{KinnunenLehrbackVahakangasZhong} we would have $\nu=\mu$.

We sketch the proof
in the light of this difference.
The first step is to write $g=C M^{\nu,p}_{\mathcal{B}_0}u$  for a suitable
constant $C=C(c_\nu)>0$ in such a way that inequality
\[
\lvert u(\gamma(0))-u(\gamma(\ell_\gamma))\rvert
\le \int_\gamma g\,ds
\]
holds for {\em every} curve  $\gamma\colon [0,\ell_\gamma]\to B_0$. The doubling condition \eqref{e.doubling} for $\nu$ and chaining arguments are used here. 
It follows that 
$g|_{B_0}$ is a $p$-weak upper gradient of $u|_{B_0}$ with respect to $B_0$.
Since $u\colon X\to \R$ is a Lipschitz function, we also see that $g|_{B_0}\in L^p(B_0;d\mu)$.
Inequality \eqref{e.desired} 
follows from the fact that the restriction $g_u|_{B_0}$ is the minimal $p$-weak upper
gradient of $u|_{B_0}$ with respect to the ball $B_0$; we refer to \cite[Lemma 2.23]{MR2867756} for further details on this localization property.
\end{proof}

\begin{proof}[Proof of Theorem \ref{t.main_thm}]
We define $\Psi(t)=C_b\min\{2^\tau,t^\tau\}$ if $t>0$. Then $\Psi\colon (0,\infty)\to (0,\infty)$ is a bounded non-decreasing function. Moreover,  by
the $(p-\tau)$-balance condition assumption,  the space  $X$ satisfies the $\Psi$-bumped $p$-balance condition as in Definition \ref{d.bumped_balance_def}.
Fix a ball $B_0\subset X$ and  let $\mathcal{B}_0$ 
be the family~\eqref{eq.ball_family_0} of balls associated with $B_0$. 
Let $u\in\Lip(X)$ and let $g_u\in\mathcal{D}^{p}(u)$ be the minimal $p$-weak upper gradient of $u$.  
Since $X$ has the $(p-\vartheta)$-independence property, it will be enough to establish 
the maximal Poincar\'e inequality with respect to $g_u$. 
Let $k\in \N$, $0\le \eps< p-1$,
 and $\alpha=p/(2(s+p))>0$ with $s=\log_2 c_\nu$. 
By Theorem \ref{t.main_local} in Section \ref{s.main},
\begin{equation}\label{e.loc_des_PRIOR}
\begin{split}
\int_{B_0}\bigl( M^{\nu,p}_{\mathcal{B}_0} u\bigr)^{p-\eps}\,d\mu
&\le C_1\biggl((2^{-k\alpha}+\Psi(C_1 2^{-k\alpha}))2^{k\eps}+\frac{K_{p,p}4^{k\eps}}{k^{p-1}}\biggr)\int_{B_0} 
\bigl( M^{\nu,p}_{\mathcal{B}_0} u\bigr)^{p-\eps}\,d\mu\\
&\qquad
+ C_1 C(k,\eps)K_{p,p}\int_{B_0\setminus \{M^{\nu,p}_{\mathcal{B}_0} u=0\}} g_u^p\bigl( M^{\nu,p}_{\mathcal{B}_0} u\bigr)^{-\eps}\,d\mu\,.
\end{split}
\end{equation}
 Here the constant
$C_1>0$ depends only on $\delta$, $p$,  $c_\mu$, $c_\nu$, $c_{\nu,\mu}$ and $\lVert \Psi \rVert_\infty$.
We also remark that the left-hand side of \eqref{e.loc_des_PRIOR} is finite,
due to the fact that $u$ is a Lipschitz function in $X$.

We now choose $k\in\N$ and $0<\eps_0<\min \{\vartheta,1/k,p-1\}$ such that
\[
C_2=C_1\biggl((2^{-k\alpha}+C_b(C_1 2^{-k\alpha})^\tau)2^{k\eps_0}+\frac{K_{p,p}4^{k\eps_0}}{k^{p-1}}\biggr)<\frac{1}{2}\,.
\]
Fix $0\le \eps\le \eps_0$. Then
\begin{align*}
&C_1\biggl((2^{-k\alpha}+\Psi(C_1 2^{-k\alpha}))2^{k\eps}+\frac{K_{p,p}4^{k\eps}}{k^{p-1}}\biggr)
\\&\qquad \qquad\qquad \le C_1\biggl((2^{-k\alpha}+C_b(C_1 2^{-k\alpha})^\tau)2^{k\eps}+\frac{K_{p,p}4^{k\eps}}{k^{p-1}}\biggr)\le C_2< \frac{1}{2}\,.
\end{align*}
This allows us to absorb the first term on the right-hand side  of \eqref{e.loc_des_PRIOR}
to the left-hand side; recall that this term is finite. By absorption and Lemma \ref{l.Lip_estimate}, we obtain that
\begin{align*}
\int_{B_0}\bigl( M^{\nu,p}_{\mathcal{B}_0} u\bigr)^{p-\eps}\,d\mu
&\le 2C_1 C(k,\eps)K_{p,p}\int_{B_0\setminus \{M^{\nu,p}_{\mathcal{B}_0} u=0\}} g_u^p\bigl( M^{\nu,p}_{\mathcal{B}_0} u\bigr)^{-\eps}\,d\mu\\
&\le  2C(c_\nu)^{\eps}C_1 C(k,\eps)K_{p,p}\int_{B_0\setminus \{M^{\nu,p}_{\mathcal{B}_0} u=0\}} g_u^{p-\eps}\,d\mu\le  C_3\int_{B_0} g_u^{p-\eps}\,d\mu\,.
\end{align*}
Here the constant 
$C_3>0$ is independent of the parameters $B_0$, $u$ and $g_u$. 
Since $X$ has the $(p-\vartheta)$-independence property  and $\eps_0<\vartheta$,  we conclude that
$X$ supports a maximal $(p,p-\eps)$-Poincar\'e inequality
for every $0\le \eps\le \eps_0$.
\end{proof}

Under certain conditions, we can now
essentially characterize when a $(p,p)$-Poincar\'e
inequality improves to a $(p,p-\eps)$-Poincar\'e
inequality. 
This characterization will be given in terms of
balance conditions.
Indeed, recalling Lemma \ref{l.straightforward}, this
is essentially the content of the following corollary,
which is among our main results.

\begin{corollary}\label{c.self_impro}
Let $1<p<\infty$ and let $0<\vartheta<p-1$. 
Assume that a geodesic two-measure space $X=(X,d,\nu,\mu)$ 
supports a $(p,p)$-Poincar\'e inequality~\ref{d.poincare}
and has the $(p-\vartheta)$-independence property \ref{d.independence}, and that $\nu$ is an
$A_\infty(\mu)$-weighted measure 
(Definition \ref{d.comparability}).   
Then the following conditions (A)--(C) are equivalent:
\begin{itemize}
\item[(A)]
$X$ satisfies  a $(p,p-\tau)$-balance condition
for some  $0<\tau<p-1$.
\item[(B)]
$X$ satisfies a $(p-\tau)$-balance condition
for some $0<\tau <p-1$. 
\item[(C)]
There exists $0<\eps_0<p-1$ such that $X$ supports a 
maximal $(p,p-\eps)$-Poincar\'e inequality for every $0\le \eps\le \eps_0$.
\end{itemize}
\end{corollary}

\begin{proof} 
The two implications (A) $\Longrightarrow$ (B) and 
(B) $\Longrightarrow$ (C)
follow from Proposition~\ref{p.basic_balance}(C) and
Theorem~\ref{t.main_thm}, respectively.
Hence, it remains to prove the implication (C) $\Longrightarrow$ (A).
By condition (C) with $0<\eps=\eps_0<p-1$ and Lemma \ref{l.straightforward}, we see that
there is a constant $C>0$ such that  inequality
\[\biggl(\vint_{B} \lvert u(x)-u_{B;\nu}\rvert^p\,d\nu(x)\biggr)^{1/p}
\le C\diam(B)\biggl(\vint_{2B} g(x)^{p-\eps} \,d\mu(x)\biggr)^{1/(p-\eps)}
\]
holds whenever $B$ is a ball in $X$ and whenever 
$u\in\Lip(X)$ and $g\in \mathcal{D}^{p-\eps}(u)$. 
By Lemma \ref{l.Poincare_implies_balance}, we find
that $X$ satisfies a $(p,p-\eps)$-balance condition.
Condition (A) follows with $\tau=\eps$.
\end{proof}

\subsection{Results for maximal Poincar\'e inequalities}\label{s.si_II}

Let $X=(X,d,\nu,\mu)$ be a 
geodesic two-measure space and let $1<p<\infty$.
In addition to maximal Poincar\'e inequalities, we will also 
consider properties of the 
global maximal function $M^{\nu,p}u=M^{\nu,p}_{\mathcal{X}}u$ with respect to the family
\begin{equation}\label{eq.global_balls}
\mathcal{X}=\{B\subset X\,:\, B\text{ is a ball}\}
\end{equation}
of all balls in $X$.
The novelty of the following result is the implication (C) $\Longrightarrow$ (D) which, in a certain sense, gives the self-improvement of 
\emph{global} maximal Poincar\'e 
inequalities without assuming any a priori balance conditions. 
The proof of this implication is based on Theorem~\ref{t.global_improvement} from  Section \ref{s.global}.
In fact, Theorem~\ref{t.global_improvement}  
yields 
this particular implication under significantly weaker assumptions than those in Theorem~\ref{t.global_epsilon_0}.

On the other hand,  the following Theorem \ref{t.global_epsilon_0} has the advantage that it provides yet further conditions that
--- under the provided stronger assumptions --- are all equivalent  to
any of the conditions (A)--(C) in Corollary \ref{c.self_impro}. Hence, Theorem \ref{t.global_epsilon_0}
both extends and complements Corollary \ref{c.self_impro}; indeed,  the assumptions of these results
are identical.

\begin{theorem}\label{t.global_epsilon_0}
Let $1<p<\infty$ and let $0<\vartheta<p-1$.
Assume that a geodesic two-measure space $X=(X,d,\nu,\mu)$ 
supports a $(p,p)$-Poincar\'e inequality~\ref{d.poincare}
and has the $(p-\vartheta)$-independence property \ref{d.independence}, and 
that $\nu$ is an
$A_\infty(\mu)$-weighted measure (Definition \ref{d.comparability}).   
Then the following conditions (A)--(D) are equivalent:
\begin{itemize}
\item[(A)] $X$ satisfies a $(p-\tau)$-balance condition for some $0<\tau<p-1$. 
\item[(B)] $X$ supports a maximal $(p,p)$-Poincar\'e inequality.
\item[(C)] There is a constant $C>0$ such that 
\begin{equation*}
\begin{split}
\int_{X}\bigl( M^{\nu,p} u\bigr)^{p}\,d\mu
\le C\int_{X} g^p\,d\mu
\end{split}
\end{equation*}
whenever $u\in\Lip(X)$ and  $g\in\mathcal{D}^p(u)$.
\item[(D)]
There are constants $0<\varepsilon<p-1$ and $C>0$ such that
\begin{equation*}
\begin{split}
\int_{X}\bigl( M^{\nu,p} u\bigr)^{p-\varepsilon}\,d\mu
\le C\int_{X} g^{p-\varepsilon}\,d\mu
\end{split}
\end{equation*}
whenever $u\in\Lip(X)$ has a bounded support and $g\in\mathcal{D}^{p-\eps}(u)$.
\end{itemize} 
\end{theorem}

\begin{proof}
The implication (A) $\Longrightarrow$ (B) is a consequence of Corollary \ref{c.self_impro}.

Consider then the implication (B) $\Longrightarrow$ (C).
Fix a point $x_0\in X$. For every $j\in\N$, we denote $B_j=B(x_0,j)$ and
\[\mathcal{B}_j=\{B(x,r)\,:\,B(x,2r) \subset B_j\}\,,\]
that is, $\mathcal{B}_j$ is the family of balls associated with  $B_j$.
Fix $u\in \Lip(X)$ and $g\in \mathcal{D}^p(u)$.
Observe that
\[
M^{\nu,p} u(x)= \lim_{j\to \infty} \bigl(\mathbf{1}_{B_j}(x) M^{\nu,p}_{\mathcal{B}_j} u(x)\bigr)
\]
whenever $x\in X$.
By Fatou's lemma and condition (B), there is a constant $C>0$, independent of $u$ and $g$, such that
\begin{align*}
\int_{X}\bigl( M^{\nu,p} u(x)\bigr)^{p}\,d\mu(x)&\le \liminf_{j\to\infty}\int_{B_j} \bigl( M^{\nu,p}_{\mathcal{B}_j} u(x)\bigr)^{p}\,d\mu(x)
\\&\le C \liminf_{j\to\infty}\int_{B_j}  g(x)^p\,d\mu(x)
\le C \int_X g(x)^p\,d\mu(x)\,.
\end{align*}
Hence, we see that condition (C) is valid.

The implication (C) $\Longrightarrow$ (D) is a consequence of Theorem~\ref{t.global_improvement} 
and the assumption that $X$ has the $(p-\vartheta)$-independence property.  Notice that we may require in Theorem~\ref{t.global_improvement} that $\eps_0$ is such that $0<\eps_0<\vartheta$. 
Theorem~\ref{t.global_improvement} then yields the claim (D) for all $g\in\mathcal{D}^p(u)$, and by the 
independence property we conclude that the claim holds also for all $g\in\mathcal{D}^{p-\eps}(u)$. 

It remains to prove the implication (D) $\Longrightarrow$ (A). Let $0<\eps<p-1$ and $C>0$ be constants as in condition (D). 
Fix $u\in\Lip(X)$ with a bounded support and $g\in\mathcal{D}^{p-\eps}(u)$, and let $B\subset X$ be a ball.  
By Definition \ref{d.m_def} of $M^{\nu,p}u=M_{\mathcal{X}}^{\nu,p}u$ 
and condition (D),  we then obtain that 
\begin{align*}
\mu(B)\biggl(\frac{1}{\diam(B)^p}\vint_{B} \lvert u(y)-u_{B;\nu}\rvert^p\,d\nu(y)\biggr)^{(p-\eps)/p}&\le \int_{B} \bigl(M^{\nu,p} u(x)\bigr)^{p-\eps}\,d\mu(x)\\
\le \int_X\bigl(M^{\nu,p} u(x)\bigr)^{p-\eps}\,d\mu(x)&\le C\int_X g(x)^{p-\eps}\,d\mu(x)\,.
\end{align*}
By Lemma \ref{l.Poincare_implies_balance}, we see that
$X$ satisfies a $(p,p-\varepsilon)$-balance condition, and thus
$X$ satisfies a $(p-\eps)$-balance condition, by Proposition~\ref{p.basic_balance}(C).  
That is, condition (A) holds with $\tau=\eps$.
\end{proof}

We can now provide the example showing that a $(p,p)$-Poincar\'e inequality does not always
imply a {\em maximal} $(p,p)$-Poincar\'e inequality.

\begin{example}\label{ex.sharp_explanation}
Let $1<p<n$ and let $X=(X,d,\nu,\mu)$ be the complete geodesic two-measure space as in Example \ref{ex.bal_no_impro}.
By  Example \ref{ex.no_impro}, $X$ supports a $(p,p)$-Poincar\'e inequality.
Moreover, the space $Y=(X,d,\mu,\mu)$ supports
a $(1,1)$-Poincar\'e inequality; cf.\ \cite[Proposition A.17]{MR2867756}. Therefore, Example~\ref{e.p_indep} shows that $X$ has
the $(p-\vartheta)$-independence property if $0<\vartheta<p-1$. 
 Since  
$A_1(\mu)\subset A_\infty(\mu)$ by~\eqref{ap_relations},  
$\nu$ is an $A_\infty(\mu)$-weighted measure. 
On the other hand, by Example \ref{ex.bal_no_impro}, we see that $X$ does not satisfy a $(p-\tau)$-balance
condition for any $0<\tau< p-1$. 
We can now apply Theorem~\ref{t.global_epsilon_0}, which
implies that $X$ does not support a maximal $(p,p)$-Poincar\'e inequality.
\end{example}

\section{Norm estimates for the sharp maximal function}\label{s.main}

Let $X=(X,d,\nu,\mu)$ be a 
geodesic two-measure space and let  
$1<p<\infty$. 
In this section we are {\em primarily} interested in the 
localized sharp maximal function $M^{\nu,p}_{\mathcal{B}_0}u$ that is associated with the ball family
\begin{equation}\label{e.b_0}
\mathcal{B}_0=\{B(x,r)\,:\,B(x,2r)\subset B_0\}\,.
\end{equation}
Here and in the statement of Theorem \ref{t.main_local}, the set $B_0\subset X$ of localization is a fixed ball, 
and the case $X=B_0$ is allowed but then $X$ is of course necessarily bounded.

\begin{theorem}\label{t.main_local}
Let $1<p<\infty$ and assume  
that $X=(X,d,\nu,\mu)$ is a geodesic two-measure space satisfying the following
assumptions: 
\begin{itemize}
\item
$X$ supports a $(p,p)$-Poincar\'e inequality
with a constant $K_{p,p}>0$ (Definition \ref{d.poincare})
\item 
$\nu$ is an $A_\infty(\mu)$-weighted measure,
with constants $c_{\nu,\mu},\delta>0$ (Definition \ref{d.comparability})
\item
There exists a bounded non-decreasing function $\Psi\colon (0,\infty)\to (0,\infty)$ such that 
$X$ satisfies the $\Psi$-bumped $p$-balance condition (Definition \ref{d.bumped_balance_def}).
\end{itemize}
Let $k\in \N$, $0\le \varepsilon< p-1$,
and $\alpha=p/(2(s+p))>0$ with $s=\log_2 c_\nu$. In addition, let 
$B_0\subset X$ be  a fixed ball and let $\mathcal{B}_0$ be the  family~\eqref{e.b_0} of balls associated with $B_0$.
Then inequality
\begin{equation}\label{e.loc_des}
\begin{split}
\int_{B_0}\bigl( M^{\nu,p}_{\mathcal{B}_0} u\bigr)^{p-\varepsilon}\,d\mu
&\le C_1\biggl((2^{-k\alpha}+\Psi(C_1 2^{-k\alpha}))2^{k\varepsilon}+
\frac{K_{p,p}4^{k\varepsilon}}{k^{p-1}}\biggr)\int_{B_0} \bigl( M^{\nu,p}_{\mathcal{B}_0} u\bigr)^{p-\varepsilon}\,d\mu\\
&\qquad+
C_1 C(k,\varepsilon)K_{p,p}\int_{B_0\setminus \{M^{\nu,p}_{\mathcal{B}_0} u=0\}} g^p\bigl( M^{\nu,p}_{\mathcal{B}_0} u\bigr)^{-\varepsilon}\,d\mu
\end{split}
\end{equation}
holds whenever $u\in\Lip(X)$ and $g\in\mathcal{D}^{p}(u)$. Here we denote
$C(k,\varepsilon)=(4^{k\varepsilon}-1)/\varepsilon$ if 
$\varepsilon>0$ and $C(k,0)=k$. Moreover, the constant
$C_1>0$ depends only on $\delta$, $p$,  $c_\mu$, $c_\nu$, $c_{\nu,\mu}$ and $\lVert \Psi \rVert_\infty$.
\end{theorem}

The proof of Theorem \ref{t.main_local} is an adaptation of a corresponding proof in \cite{KinnunenLehrbackVahakangasZhong} for the case of a single measure
$\mu=\nu$. 
However, many of the changes needed for 
the present setting of two measures are somewhat technical. For this reason, and also in order to make this work relatively
self-contained, we 
provide most of the  details below.

The proof of Theorem \ref{t.main_local} is completed in Section \ref{ss.main_local}. 
For the proof, we need preparations
that are treated in Sections~\ref{ss.Whitney} -- \ref{ss.auxiliary_local}.
At this stage, we already fix  $X=(X,d,\nu,\mu)$,  $p$, $K_{p,p}$,  $c_{\nu,\mu}$, $\delta$, $\Psi$, $k$, $\varepsilon$, $\alpha$, $s$, 
$B_0\subsetneq X$, $\mathcal{B}_0$,  and $u$ as
in the statement of Theorem \ref{t.main_local}.
We refer to these objects throughout 
Section \ref{s.main}
without further notice. 
However, the $p$-weak upper gradient $g\in\mathcal{D}^p(u)$ is not yet fixed at this stage.

Let us emphasize that the ball $B_0$ in the arguments below is further assumed to be a strict subset of $X$. 
That is, we will only focus on the case $B_0\not=X$. However, if $B_0=X$, then $X$ is bounded and the following Whitney cover 
$\mathcal{W}_0$ can be replaced with the singleton $\{Q=B_0\}$. The other modifications in this easier special case are straightforward and we omit the details.

\subsection{Whitney ball  cover}\label{ss.Whitney} 

We will need a {\em Whitney ball cover} 
$\mathcal{W}_0=\mathcal{W}(B_0)$ of  $B_0\subsetneq X$.
This countable family with good covering properties
is comprised of the so-called {\em Whitney balls} that are of the form $Q=B(x_Q,r_Q)\in\mathcal{W}_0$,
with center $x_Q\in B_0$ and radius
\[
r_Q=\frac{\dist(x_Q, X\setminus B_0)}{128}>0\,.
\]
The $4$-dilated Whitney ball is denoted by $Q^*=4Q=B(x_Q,4 r_Q)$ whenever $Q\in\mathcal{W}_0$.
Even though the Whitney balls need not be pairwise disjoint, they nevertheless have 
the following standard covering properties with bounded overlap; cf.\ \cite[pp.~77--78]{MR2867756}:
\begin{itemize}
\item[(W1)] $B_0=\bigcup_{Q\in\mathcal{W}_0} Q$;
\item[(W2)] $\sum_{Q\in\mathcal{W}_0} \mathbf{1}_{Q^*}\le C\mathbf{1}_{B_0}$ for some constant $C=C(c_\nu)>0$.
\end{itemize}
The facts (W3)--(W6) below 
for any Whitney ball $Q=B(x_Q,r_Q)\in\mathcal{W}_0$
are straightforward to verify
by using inequality \eqref{e.diams} and the assumption $B_0\subsetneq X$; we omit the simple proofs.
Here 
$\mathcal{B}_0$ is the  family~\eqref{e.b_0} of balls associated with the fixed ball $B_0$.
\begin{itemize} 
\item[(W3)] If $B\subset X$ is a ball such that $B\cap Q\not=\emptyset\not=2B\cap(X\setminus Q^*)$, then
$\diam(B)\ge 3r_Q/4$.
\item[(W4)] If  $B\subset Q^*$ is a ball, then $B\in\mathcal{B}_0$.
\item[(W5)] If $B\subset Q^*$ is a ball, $x\in B$  and $0<r\le \diam(B)$, then
$B(x,5r)\in\mathcal{B}_0$.
\item[(W6)] If $x\in Q^*$ and $0<r\le 2\diam(Q^*)$, then $B(x,r)\in\mathcal{B}_0$.
\end{itemize}
Observe that there is some overlap between
the conditions (W4)--(W6). The slightly different formulations will conveniently 
guide the reader
in the sequel.

\subsection{Auxiliary maximal functions}

We abbreviate
$M^{\nu} u=M^{\nu,p}_{\mathcal{B}_0} u$
and denote \[U^{\lambda}=\{x\in B_0\,:\,M^{\nu} u(x)>\lambda\}\,,\qquad \lambda>0\,.\]
The sets $U^\lambda$ are open in $X$.
If $E\subset X$ is a Borel set and $\lambda>0$, we write $U^\lambda_E=U^\lambda \cap E$.
The following lemma
is \cite[Lemma 4.12]{KinnunenLehrbackVahakangasZhong},
which in turn is a variant of \cite[Lemma 3.6]{MR1681586}.

\begin{lemma}\label{l.arm_local}
Fix $\lambda>0$ and $Q\in\mathcal{W}_0$.  
Then inequality \[
\lvert u(x)-u(y)\rvert \le C(c_\nu)\,\lambda\, d(x,y)\]
holds whenever $x,y\in  Q^*\setminus U^{\lambda}$.
\end{lemma}

We also need a certain smaller maximal function that is localized to Whitney balls.
More specifically, for each  $Q\in\mathcal{W}_0$, we first introduce the ball family\footnote{It is important to use condition `$B\subset Q^*$' in the 
definition for $\mathcal{B}_Q$ instead of `$B\subset Q$'.
}
\[
\mathcal{B}_{Q}=\{B\subset X\,:\,  B  \text{ is a ball such that }B\subset Q^*\}
\]
and then define  $M^{\nu}_{Q} u=M^{\nu,p}_{\mathcal{B}_{Q}} u$.
By using these individual maximal functions, we then define a {\em Whitney ball localized maximal function}\footnote{It is
equally important to use $\mathbf{1}_Q$ instead of $\mathbf{1}_{Q^*}$ in the
definition of $M^{\nu}_{\mathrm{loc}} u$; these are delicate matters and related to the latter selection of stopping balls with the aid of condition (W3).}
\[
M^{\nu}_{\textup{loc}} u = \sup_{Q\in\mathcal{W}_0} \mathbf{1}_Q M^{\nu}_{Q} u\,.
\]
If $\lambda>0$ and $Q\in\mathcal{W}_0$, we write \begin{equation}\label{e.super}
Q^\lambda = \{x\in Q\,:\, M^{\nu}_Q u(x)>\lambda\}\,,\qquad V^{\lambda}=\{x\in B_0\,:\, M^{\nu}_{\textup{loc}}u(x)>\lambda\}\,.
\end{equation}

The following Lemma \ref{l.big_to_small_ball}
provides a norm estimate between the different maximal functions.
Its purpose, roughly speaking, is to create space for the stopping balls in Section~\ref{s.stopping} to expand, without losing their control in terms of $M^{\nu} u$.
Controlling this expansion  is the only purpose for introducing the different maximal function aside from $M^{\nu} u =M^{\nu,p}_{\mathcal{B}_0} u$.

\begin{lemma}\label{l.big_to_small_ball}
There is a constant $C=C(p,c_\nu,c_\mu)\ge 1$  such that
\[
\int_{B_0} \bigl(M^{\nu} u(x)\bigr)^{p-\varepsilon}\,d\mu(x)\le C\int_{B_0} \bigl(M^{\nu}_{\textup{loc}} u(x)\bigr)^{p-\varepsilon}\,d\mu(x)\,.
\]
\end{lemma}

Lemma \ref{l.big_to_small_ball} is a two-measure variant of
\cite[Lemma 4.10]{KinnunenLehrbackVahakangasZhong}. The key ingderient in the proof is a 
modification of a distributional inequality 
\cite[Lemma 12.3.1]{MR3363168};  see also  \cite[Lemma 3.2.1]{MR2415381}.
The two measures $\nu$ and $\mu$ do not significantly interact
in the proof of Lemma \ref{l.big_to_small_ball}, and hence we omit the straightforward but 
tedious  modifications that are needed for the proof.

\subsection{Stopping construction}\label{s.stopping}
The following stopping construction is needed for each Whitney ball separately.
Fix $Q\in\mathcal{W}_0$.
The number
\[
\lambda_Q=\biggl(\frac{1}{\diam(Q^*)^{p}} \vint_{Q^*} \lvert u(y)-u_{Q^*;\nu}\rvert^p\,d\nu(y)\biggr)^{1/p}
\]
serves as a certain treshold value.
Fix a level $\lambda>\lambda_Q/2$.
We will construct a stopping family $\mathcal{S}_\lambda(Q)$ of balls
whose $5$-dilations, in particular, cover the set $Q^\lambda$; recall the definition~\eqref{e.super}.
As a first step towards the stopping balls, let
$B\in\mathcal{B}_{Q}$ be such that $B\cap Q\not=\emptyset$. The {\em parent ball} of $B$ is then defined to be $\pi(B)=2B$ if $2B\subset Q^*$ and $\pi(B)=Q^*$ otherwise.
Observe that since  $B\subset \pi(B)\in\mathcal{B}_Q$ and $\pi(B)\cap Q\not=\emptyset$, 
the grandparent $\pi(\pi(B))$ is well defined, and so on and so forth.
Moreover, by inequalities \eqref{e.doubling} and \eqref{e.diams}, and property (W3) if needed, we have $\nu(\pi(B))\le c_\nu^5 \nu(B)$
and $\diam(\pi(B))\le 16\diam(B)$.

Now we come to the actual stopping argument. We fix  $x\in Q^\lambda\subset Q$.
If $\lambda_Q/2<\lambda<\lambda_Q$, then we choose
$B_x=Q^*\in\mathcal{B}_{Q}$. If 
$\lambda\ge \lambda_Q$, then by using the condition $x\in Q^\lambda$ we first choose a starting ball $B$, with $x\in B\in\mathcal{B}_Q$, such that
\[
\lambda <
\biggl(\frac{1}{\diam(B)^{p}} \vint_{B} \lvert u(y)-u_{B;\nu}\rvert^p\,d\nu(y)\biggr)^{1/p}\,.
\]
We continue by looking at the balls $B\subset \pi(B) \subset \pi(\pi(B))\subset \dotsb$ 
and we 
stop at the first ball among them, denoted by $B_x\in\mathcal{B}_{Q}$, that satisfies the following stopping conditions: 
\begin{align*}
\begin{cases} 
\lambda <
\displaystyle\biggl(\frac{1}{\diam(B_x)^{p}} \vint_{B_x} \lvert u(y)-u_{B_x;\nu}\rvert^p\,d\nu(y)\biggr)^{1/p},\\
\displaystyle \biggl(\frac{1}{\diam(\pi(B_x))^{p}} \vint_{\pi(B_x)} \lvert u(y)-u_{\pi(B_x);\nu}\rvert^p\,d\nu(y)\biggr)^{1/p}\le \lambda \,.
 \end{cases}
\end{align*}
The inequality $\lambda\ge \lambda_Q$ in combination with the assumption $B_0\subsetneq X$ ensures that there always exists such a stopping ball.
In both cases above,  the chosen ball $B_x^\lambda=B_x\in\mathcal{B}_Q$  contains
the point $x$ and satisfies  the  inequalities
\begin{equation}\label{e.loc_stop}
\lambda<
\biggl(\frac{1}{\diam(B_x^\lambda)^{p}} \vint_{B_x^\lambda} \lvert u(y)-u_{B_x^\lambda;\nu}\rvert^p\,d\nu(y)\biggr)^{1/p}\le 32c_\nu^{5/p} \lambda\,.
\end{equation}
Now, by using the $5r$-covering lemma,  we obtain a countable pairwise disjoint family 
\[\mathcal{S}_\lambda(Q)\subset \bigl\{B_x^\lambda\,:\, x\in Q^\lambda\bigr\}\,,\qquad \lambda>\lambda_Q/2\,,\]
of {\em stopping balls}  such that $Q^\lambda\subset \bigcup_{B\in\mathcal{S}_\lambda(Q)}
5B$.
Let us remark that, by the condition (W4) and stopping inequality \eqref{e.loc_stop}, we have $B\subset U^{\lambda}_{Q^*}=U^{\lambda}\cap Q^*$ 
if
$B\in \mathcal{S}_\lambda(Q)$ and $\lambda>\lambda_Q/2$.

\subsection{Auxiliary local results}\label{ss.auxiliary_local}

This section contains 
two technical results: Lemma \ref{l.dyadic} and Lemma \ref{l.mainl_local}.
Even though these lemmata are slight variants of their
counterparts in \cite{KinnunenLehrbackVahakangasZhong},
we provide all the details here. In particular, we develop novel comparison
and balancing arguments  which show how the measures $\nu$ and $\mu$ 
need to interact.

Lemma \ref{l.dyadic} below  
is the only place in the proof of Theorem~\ref{t.main_local} 
where the $\Psi$-bumped $p$-balance condition is needed.
This lemma is a two-measure counterpart of 
\cite[Lemma 4.15]{KinnunenLehrbackVahakangasZhong};
see also \cite[Lemma 3.1.2]{MR2415381}. 
Recall that 
\[\alpha=p/(2(s+p))>0\,,\qquad \text{ with } s=\log_2 c_\nu> 0\,.\]

\begin{lemma}\label{l.dyadic}
Let
$Q\in\mathcal{W}_0$ be  a Whitney ball and let $\lambda>\lambda_Q/2$. Then there exists a constant
$C_1=C(\delta,p,c_\nu,c_{\nu,\mu},\lVert \Psi\rVert_\infty)>0$  such that inequality
\begin{equation}\label{e.abso}
\begin{split}
\frac{1}{\diam(B)^{p}}&\int_{U_{B}^{2^k\lambda}}\lvert u(x)-u_{B\setminus U^{2^k\lambda};\nu}\rvert^p\,d\nu(x)\\&\le C_1( 2^{-k\alpha}+\Psi(C_12^{-k\alpha}))(2^{k}\lambda)^p 
\mu(U_{B}^{2^k\lambda})\frac{\nu(B)}{\mu(B)}
\\&\quad\qquad\qquad+\frac{C_1}{\diam(B)^{ p}}\int_{B\setminus U^{2^k\lambda}} \lvert u(x)-u_{B\setminus U^{2^k\lambda};\nu}\rvert^p\,d\nu(x)
\end{split}
\end{equation}
holds whenever $B\in\mathcal{S}_\lambda(Q)$
is such that $\nu (U_{B}^{2^k\lambda}) < \nu(B)/2$.
\end{lemma}

\begin{proof}
Fix $\lambda>\lambda_Q/2$ and  let $B\in\mathcal{S}_\lambda(Q)$ be such that
$\nu (U_{B}^{2^k\lambda}) < \nu(B)/2$. Fix $x\in U_{B}^{2^k\lambda}\subset B$. Consider
the function $h\colon (0,\infty)\to\R$, 
\[
r\mapsto h(r)= \frac{\nu(U_{B}^{2^k\lambda}\cap B(x,r))}{\nu(B\cap B(x,r))}=\frac{\nu(U_{B}^{2^k\lambda}\cap B(x,r))}{\nu(B(x,r))}\cdot 
\biggl(\frac{\nu(B\cap B(x,r))}{\nu(B(x,r))}\biggr)^{-1}\,.
\]
By Lemma \ref{l.continuous} and the fact that $B$ is open,  $h$   
is continuous.
Since $h(r)=1$ for small values of $r>0$
and $h(r)<1/2$ for $r>\diam(B)$, 
there is $0<r_x\le \diam(B)$ such that $h(r_x)=1/2$.
Write  $B'_x=B(x,r_x)$. Then
\begin{equation}\label{e.tok}
\frac{\nu(U_{B}^{2^k\lambda}\cap  B'_x)}{\nu(B\cap B'_x)}=h(r_x)=\frac{1}{2}
\end{equation}
and
\begin{equation}\label{e.ens}
\frac{\nu((B\setminus U^{2^k\lambda})\cap  B'_x)}{\nu(B\cap  B'_x)}
=1-\frac{\nu(U_{B}^{2^k\lambda}\cap  B'_x)}{\nu(B\cap B'_x)}
= 1-h(r_x)=\frac{1}{2}\,.
\end{equation}
Let $\mathcal{G}_\lambda$ be a countable and pairwise disjoint subfamily of $\{ B'_x:  x\in U_{B}^{2^k\lambda}\}$
such that $U_{B}^{2^k\lambda}\subset\bigcup_{B'\in\mathcal{G}_\lambda}5B'$.
Then 
\eqref{e.tok} and \eqref{e.ens}  hold for every ball $B'\in
\mathcal{G}_\lambda$; indeed, 
by denoting $B'_I=U_{B}^{2^k\lambda}\cap B'$ and
${B'_O}=(B\setminus U^{2^k\lambda})\cap B'$,
we have the following transition identities:
\begin{equation}\label{e.comparison}
\nu(B'_I)=  \frac{\nu( B\cap B')}{2}=  
\nu({B'_O})\,,
\end{equation}
where all the measures are strictly positive. These identities
facilitate a transition of the domain of integration 
on the left-hand side of \eqref{e.abso}
from $U^{2^k\lambda}_B$ to $B\setminus U^{2^k\lambda}_B$, and they are used several times in the proof below.

Now, we multiply the left-hand side of \eqref{e.abso}
by $\diam(B)^{ p}$ and then estimate as follows:
\begin{equation}\label{e.prepare}
\begin{split}
\int_{U_{B}^{2^k\lambda}}  &\lvert u-u_{B\setminus U^{2^k\lambda};\nu}\rvert^p\,d\nu
\le\sum_{B'\in\mathcal{G}_\lambda} \int_{5B'\cap B}\lvert u-u_{B\setminus U^{2^k\lambda};\nu}\rvert^p\,d\nu\\
&\le 2^{p-1}\sum_{B'\in\mathcal{G}_\lambda} \nu(5B'\cap B) \lvert u_{{B'_O};\nu}-u_{B\setminus U^{2^k\lambda};\nu}\rvert^p+
2^{p-1}\sum_{B'\in\mathcal{G}_\lambda} \int_{5B'\cap B}\lvert u-u_{{B'_O};\nu}\rvert^p\,d\nu\,.
\end{split}
\end{equation}
By \eqref{e.doubling} and Lemma \ref{l.ball_measures},
we find that
 $\nu(5B'\cap B)\le \nu(8B') \le c_\nu^6 \nu(B\cap B')$
 if $B'\in\mathcal{G}_\lambda$.
Hence, by the transition identities \eqref{e.comparison},
 \begin{equation}\label{e.eka}
\begin{split}
2^{p-1}&\sum_{B'\in\mathcal{G}_\lambda} \nu(5B'\cap B)  \lvert u_{{B'_O};\nu}-u_{B\setminus U^{2^k\lambda};\nu}\rvert^p
\le C(c_\nu,p) \sum_{B'\in\mathcal{G}_\lambda} \nu({B'_O})  
\vint_{{B'_O}} \lvert u-u_{B\setminus U^{2^k\lambda};\nu}\rvert^p\,d\nu
\\&=C(c_\nu,p)\sum_{B'\in\mathcal{G}_\lambda}  
\int_{{B'_O}} \lvert u-u_{B\setminus U^{2^k\lambda};\nu}\rvert^p\,d\nu
\le C(c_\nu,p)
\int_{B\setminus U^{2^k\lambda}} \lvert u-u_{B\setminus U^{2^k\lambda};\nu}\rvert^p\,d\nu\,.
\end{split}
\end{equation}
This concludes our analysis of the `easy term'  on the right-hand side of~\eqref{e.prepare}. 
In order to treat the remaining term therein, we  need some preparations.

Fix a ball $B'\in\mathcal{G}_\lambda$ that satisfies
$\int_{5B'\cap B} \lvert u-u_{{B'_O};\nu}\rvert^p\,d\nu\not=0$. 
By using Lemma \ref{l.ball_measures} and the transition identities~\eqref{e.comparison}, 
we have $\nu(B')\le c_\nu^3\nu(B\cap B')= 2c_\nu^3 \nu(B'_I)$. Hence, from the assumption that
$\nu$ is an $A_\infty(\mu)$-weighted measure, it follows that
\begin{equation}\label{e.in}
\mu(B')\le C(\delta,c_\nu,c_{\nu,\mu}) \mu(B'_I)\,.
\end{equation}
We also claim that
\begin{equation}\label{e.out}
\vint_{5B'\cap B}\lvert u-u_{{B'_O};\nu}\rvert^p\,d\nu\le C_2(2^{-k\alpha}+\Psi(C_22^{-k\alpha})) (2^k\lambda)^p \diam(B)^{p}\frac{\nu(B)}{\nu(B')}\frac{\mu(B')}{\mu(B)}
\end{equation}
with a constant $C_2=C(p,c_\nu,\lVert \Psi\rVert_\infty)$. 
In order to prove this inequality, we fix a number $m\in \R$ such that 
\begin{align*}
(2^m \lambda)^p \diam(5B')^{p}&=\vint_{5B'\cap B}\lvert u-u_{{B'_O};\nu}\rvert^p\,d\nu\,.
\end{align*}
We proceed with a case study.
If $m< k/2$, then $m-k<-k/2$, and since always $\alpha<p/2$, inequality \eqref{e.out} is obtained in this case
by using inequality \eqref{e.diams} and the assumed $\Psi$-bumped $p$-balance condition as follows: 
\begin{align*}
\vint_{5B'\cap B}\lvert u-u_{{B'_O};\nu}\rvert^p\,d\nu
&=2^{(m-k)p} (2^k\lambda)^p\diam(5B')^{p} 
\\&\le 4^{3p}\lVert \Psi\rVert_\infty2^{-k\alpha}(2^k\lambda)^p\diam(B)^p\frac{\nu(B)}{\nu(B')}\frac{\mu(B')}{\mu(B)}\,.
\end{align*}

Next we consider the case $k/2\le m$.
By the transition identities \eqref{e.comparison} and Lemma \ref{l.ball_measures},
\begin{align*}
\vint_{5B'\cap B} \lvert u-u_{{B'_O};\nu}\rvert^p\,d\nu &\le  
2^{p-1}\vint_{5B'\cap B} \lvert u-u_{5B';\nu}\rvert^p\,d\nu + 2^{p-1}\lvert 
u_{5B';\nu}-u_{{B'_O};\nu}\rvert^p
\\&\le 2^{p+1}c_\nu^6 \vint_{5B'} \lvert u-u_{5B';\nu}\rvert^p\,d\nu
\le 2^{p+1}c_\nu^6 (2^k\lambda)^p\diam(5B')^{p}\,,
\end{align*}
where the last step follows from condition (W5) and the fact that $5B'\supset {B'_O}\not=\emptyset$.
From the choice of $m$ we conclude that 
$2^{mp}\le 2^{p+1} c_\nu^6 2^{kp}$.
On the other hand, we have
\begin{align*}
(2^m \lambda)^p \diam(5B')^{ p}\nu(B'\cap B)&\le 
\int_{5B'\cap B} \lvert u-u_{{B'_O};\nu}\rvert^p\,d\nu\\
&\le  2^{p-1}\int_{5B'\cap B} \lvert u-u_{B;\nu}\rvert^p\,d\nu
+ 2^{p-1}\nu(5B'\cap B) \lvert u_{{B'_O};\nu}-u_{B;\nu}\rvert^p\\
&\le 2^{p+1} c_\nu^6 \int_B \lvert u-u_{B;\nu}\rvert^p\,d\nu\\
&\le 2 \cdot 64^p c_\nu^{11} \lambda^p \diam(B)^{p}\nu(B)\,,
\end{align*}
where the last step follows from the fact that $B\in\mathcal{S}_\lambda(Q)$ in combination with
inequality \eqref{e.loc_stop}.
In particular, since $s=\log_2 c_\nu$, by inequality \eqref{e.radius_measure}
and Lemma \ref{l.ball_measures} we obtain that
\begin{align*}
\biggl(\frac{\diam(5B')}{\diam(B)}\biggr)^{s+p} &\le 20^s \frac{\diam(5B')^{p}\nu(B')}{\diam(B)^{p}\nu(B)} 
\le 20^s c_\nu^3\frac{\diam(5B')^{p}\nu(B'\cap B)}{\diam(B)^{p}\nu(B)} \\
& \le  2\cdot 64^p 20^s c_\nu^{14} 2^{-mp}\le 2\cdot 64^p 20^s c_\nu^{14} 2^{-kp/2}\,.
\end{align*}
This,  together with inequality \eqref{e.diams} and the $\Psi$-bumped $p$-balance condition,  implies that
\begin{align*}
\biggl(\frac{\diam(5B')}{\diam(B)}\biggr)^{p}
&\le 4^{3p}\biggl(\frac{\diam(B')}{\diam(B)}\biggr)^{p}
\\&\le 4^{3p}\Psi\biggl(\frac{\diam(B')}{\diam(B)}\biggr)
\frac{\nu(B)}{\nu(B')}\frac{\mu(B')}{\mu(B)}\\
&\le 4^{3p}\Psi( C(c_\nu,p)2^{-k\alpha})\frac{\nu(B)}{\nu(B')}\frac{\mu(B')}{\mu(B)}\,;
\end{align*}
here we also used the facts that $\alpha=p/(2(s+p))$ and
that the function $\Psi$ is non-decreasing.
Combining the above estimates, we see that
\begin{align*}
\vint_{5B'\cap B}\lvert u-u_{{B'_O};\nu}\rvert^p\,d\nu&=(2^m \lambda)^p \diam(5B')^{p}\\&
\le C(c_\nu,p) \Psi( C(c_\nu,p)2^{-k\alpha})(2^k\lambda)^p \diam(B)^{p}\frac{\nu(B)}{\nu(B')}\frac{\mu(B')}{\mu(B)}\,.
\end{align*}
That is, inequality \eqref{e.out} holds also in the present case $k/2\le m$.

By using Lemma \ref{l.ball_measures} followed by inequalities \eqref{e.in} and  \eqref{e.out}, 
we can now estimate the second term  on the right-hand side of~\eqref{e.prepare}  as follows, 
with a constant $C_1=C(\delta,p,c_\nu,c_{\nu,\mu},\lVert \Psi\rVert_\infty)>0$:
\begin{align*}
2^{p-1}\sum_{B'\in\mathcal{G}_\lambda} &\int_{5B'\cap B}\lvert u-u_{{B'_O};\nu}\rvert^p\,d\nu
\le C(c_\nu,p)\sum_{B'\in\mathcal{G}_\lambda} \nu(B') \vint_{5B'\cap B}\lvert u-u_{{B'_O};\nu}\rvert^p\,d\nu\\
&\le C_1(2^{-k\alpha}+\Psi(C_12^{-k\alpha}))(2^k\lambda)^p \diam(B)^{p}\frac{\nu(B)}{\mu(B)}\sum_{B'\in\mathcal{G}_\lambda} \mu(B'_I)  
\\&\le C_1(2^{-k\alpha}+\Psi(C_12^{-k\alpha}))(2^k\lambda)^p \diam(B)^{p}\mu(U^{2^k\lambda}_B)\frac{\nu(B)}{\mu(B)}\,. 
\end{align*}
  Inequality \eqref{e.abso} follows by collecting the above estimates.
\end{proof}

Whereas the $\Psi$-bumped $p$-balance condition  was only  needed in the proof of
Lemma \ref{l.dyadic}, the following lemma is the only place in the proof of Theorem \ref{t.main_local} where
the $(p,p)$-Poincar\'e inequality is needed. Moreover, 
it is invoked one single time.

\begin{lemma}\label{l.mainl_local}
Let $Q\in\mathcal{W}_0$ be a Whitney ball. Then
inequality 
\begin{equation}\label{e.id}
\begin{split}
\lambda^p \mu(Q^\lambda)
\le C_1\biggl[(2^{-k\alpha}&+\Psi(C_12^{-k\alpha}))(2^{k}\lambda)^p \mu(U^{2^k \lambda}_{Q^*})\\
& + \frac{K_{p,p}}{k^p} \sum_{j=k}^{2k-1}  (2^{j}\lambda)^p \mu(U^{2^j \lambda}_{Q^*})
  + K_{p,p}\int_{U^{\lambda}_{Q^*}\setminus U^{4^k\lambda}} g^p\,d\mu\biggr]
\end{split}
\end{equation}
holds for each $\lambda>\lambda_Q/2$ and every $g\in\mathcal{D}^{p}(u)$.
Here $C_1=C(\delta,p,c_\mu,c_\nu,c_{\nu,\mu},\lVert \Psi\rVert_\infty)>0$.
\end{lemma}

\begin{proof}
Fix $\lambda>\lambda_Q/2$ and $g\in\mathcal{D}^{p}(u)$.
By the doubling condition \eqref{e.doubling},
\[\lambda^p \mu(Q^\lambda)
\le \lambda^p \sum_{B\in\mathcal{S}_\lambda(Q)}
\mu(5B) \le c_\mu^3 \sum_{B\in\mathcal{S}_\lambda(Q)} \lambda^p \mu(B)\,.\]
Recall also that $B\subset U^{\lambda}_{Q^*}=U^\lambda\cap Q^*$ if $B\in\mathcal{S}_\lambda(Q)$.
Hence, by the fact that $\mathcal{S}_\lambda(Q)$ is a pairwise disjoint family, it suffices
to prove that 
there is a constant 
$C_2=C(\delta,p,c_\nu,c_{\nu,\mu},\lVert \Psi\rVert_\infty)>0$
such that
inequality
\begin{equation}\label{e.local}
\begin{split}
\lambda^p \mu(B)
\le C_2\biggl[( 2^{-k\alpha} & +\Psi(C_22^{-k\alpha}))(2^{k}\lambda)^p \mu(U^{2^k \lambda}_{B})\\
& +\frac{K_{p,p}}{k^p} \sum_{j=k}^{2k-1}  (2^{j}\lambda)^p \mu(U_B^{2^j \lambda})
  + K_{p,p}\int_{B\setminus U^{4^k\lambda}} g^p\,d\mu\biggl]
\end{split}
\end{equation}
holds for every $B\in\mathcal{S}_\lambda(Q)$.
To this end, let us fix a ball $B\in\mathcal{S}_\lambda(Q)$.

If $\mu(U_B^{2^k\lambda})\ge (2c_{\nu,\mu})^{-1/\delta}\mu(B)$, then 
\[\begin{split}
\lambda^p \mu(B) & \le (2c_{\nu,\mu})^{1/\delta}\lambda^p \mu(U_B^{2^k\lambda})
=C(\delta,c_{\nu,\mu})\frac{(\lambda 2^k)^p}{2^{kp}}\mu(U_B^{2^k\lambda})\\
& \le C(\delta,c_{\nu,\mu})\frac{(\lambda 2^k)^p}{2^{k \alpha}}\mu(U_B^{2^k\lambda})\,,
\end{split}\] 
which suffices for the required local estimate \eqref{e.local}.
Let us then consider the more difficult case $\mu(U_B^{2^k\lambda}) < (2c_{\nu,\mu})^{-1/\delta}\mu(B)$.
The assumption that $\nu$ is an $A_\infty(\mu)$-weighted measure 
implies that $\nu(U_B^{2^k\lambda}) < \nu(B)/2$.
By the 
stopping inequality \eqref{e.loc_stop},
\begin{align*}
\lambda^p\nu(B)&\le \frac{1}{\diam(B)^{p}}\int_{B} \lvert u(x)-u_{B;\nu}\rvert^p\,d\nu(x)
\\&\le \frac{2^{p}}{\diam(B)^{p}}\int_{X} \Bigl( \mathbf{1}_{B\setminus U^{2^k\lambda}}(x)+
\mathbf{1}_{U^{2^k\lambda}_{B}}(x)\Bigr)\lvert u(x)-u_{B\setminus U^{2^k\lambda};\nu}\rvert^p\,d\nu(x)\,.
\end{align*}
Let us emphasize that the measure term on the left-hand side above
is $\nu(B)$ instead of $\mu(B)$. The
$\Psi$-bumped $p$-balance condition and the $(p,p)$-Poincar\'e inequality
will be applied in order to transform this term, and thereby we will eventually end up with $\mu(B)$ 
on the left-hand side of \eqref{e.local}.
Actually, the $\psi$-bumped $p$-balance condition was already 
applied in Lemma \ref{l.dyadic}. By that lemma, it clearly suffices to estimate 
the integral over the set 
$B\setminus U^{2^k\lambda}=B\setminus U^{2^k\lambda}_B$;
observe that the $\nu$-measure of this set is strictly positive.

Fix a number $i\in\N$.
Recall that $B\subset Q^*$. Hence, it follows from Lemma \ref{l.arm_local} that the restriction $u|_{B\setminus U^{2^i\lambda}}\colon B\setminus U^{2^i\lambda}\to \R$ is a 
Lipschitz function with a constant $\kappa_i=C(c_\nu)2^i\lambda$. 
We  use  the McShane extension \eqref{McShane} and extend $u|_{B\setminus U^{2^i\lambda}}$ 
to a function $u_{2^i \lambda}\colon X\to \R$
that is $\kappa_i$-Lipschitz
and satisfies the restriction identity 
$u_{2^i \lambda}|_{B\setminus U^{2^i \lambda}}=u|_{B\setminus U^{2^i\lambda}}$.

We  now  define $h(x)=\frac{1}{k} \sum_{i=k}^{2k-1} u_{2^i \lambda}(x)$ whenever $x\in X$.
By conditions (D1)--(D3) in Section \ref{s.two_measure_PI}, we obtain that
\[
{\hat g}=\frac{1}{k}\sum_{i=k}^{2k-1} \Bigl( \kappa_i  \mathbf{1}_{U^{2^i\lambda}\cup B^c}
+ g\mathbf{1}_{B\setminus U^{2^i\lambda}}\Bigr)\in\mathcal{D}^{p}(h)\,.
\]
Observe that $U^{2^k\lambda}_B\supset U^{2^{(k+1)}\lambda}_B
\supset \dotsb \supset U^{2^{(2k-1)}\lambda}_B\supset U^{4^{k}\lambda}_B$.
By applying these inclusions it is straightforward to show that the following
pointwise estimates are valid in $X$,
\begin{equation*}\label{e.hardy}
\begin{split}
\mathbf{1}_B{\hat g}^p &\le \biggl( \frac{1}{k}\sum_{i=k}^{2k-1} \Bigl(   \kappa_i\,\mathbf{1}_{U_B^{2^i\lambda}}
+  g \mathbf{1}_{B\setminus U^{2^i\lambda}}\Bigr)\biggr)^p\\
&\le 2^{p}\biggl(\frac{1}{k}\sum_{i=k}^{2k-1} \kappa_i\, \mathbf{1}_{U_B^{2^i \lambda}}\biggr)^p
+  2^{p} g^p \mathbf{1}_{B\setminus U^{4^k\lambda}}\\
&\le \frac{C(c_\nu,p)}{k^p} \sum_{j=k}^{2k-1} \biggl(\sum_{i=k}^j  2^i \lambda\biggr)^p  \mathbf{1}_{U_B^{2^j \lambda}}
+  2^p g^p \mathbf{1}_{B\setminus U^{4^k\lambda}}\\
&\le \frac{C(c_\nu,p)}{k^p} \sum_{j=k}^{2k-1}  (2^{j}\lambda)^p  \mathbf{1}_{U_B^{2^j \lambda}}
+  2^p g^p \mathbf{1}_{B\setminus U^{4^k\lambda}}\,.
\end{split}
\end{equation*}
Now $h\in \Lip(X)$ coincides with $u$ on $B\setminus U^{2^k\lambda}$. Recall also that ${\hat g}\in\mathcal{D}^{p}(h)$.
Hence, by
the assumed $(p,p)$-Poincar\'e inequality,
\begin{align*}
\frac{1}{\diam(B)^{p}}&\int_{B\setminus U^{2^k\lambda}}\lvert u(x)-u_{B\setminus U^{2^k\lambda};\nu}\rvert^p\,d\nu(x)\\
&\le \frac{2^{p}}{\diam(B)^{p}}\int_{B} \lvert h(x)-h_{B;\nu}\rvert^p\,d\nu(x)
 \le 2^{p}K_{p,p}\frac{\nu(B)}{\mu(B)}\int_{B} {\hat g}(x)^p \,d\mu(x)\\
&\le \frac{C(c_\nu,p)K_{p,p}}{k^p}\frac{\nu(B)}{\mu(B)} \sum_{j=k}^{2k-1}  (2^{j}\lambda)^p \mu(U_B^{2^j \lambda})
+4^p K_{p,p}\frac{\nu(B)}{\mu(B)}\int_{B\setminus U^{4^k\lambda}} g(x)^p\,d\mu(x)\,.
\end{align*}
The desired local inequality \eqref{e.local} follows by combining the estimates above.
\end{proof}

\subsection{Completing the proof of Theorem \ref{t.main_local}}\label{ss.main_local}

Recall that  $u\colon X\to \R$ is a Lipschitz function and
 that $M^{\nu} u=M^{\nu,p}_{\mathcal{B}_0}u$.
Fix a function $g\in\mathcal{D}^{p}(u)$.
Observe that the left-hand side of 
inequality \eqref{e.loc_des} is finite since $u\in \Lip(X)$. Without loss of generality, we may further assume
that it is nonzero.
By Lemma \ref{l.big_to_small_ball},
\[
\int_{B_0} \bigl(M^{\nu} u(x)\bigr)^{p-\varepsilon}\,d\mu(x)
\le C(p,c_\nu,c_\mu)\int_{B_0} \bigl(M^{\nu}_{\textup{loc}} u(x)\bigr)^{p-\varepsilon}\,d\mu(x)\,.
\]
Observe  also  that
\[
\bigl(M^{\nu}_{\textup{loc}} u(x)\bigr)^{p-\varepsilon}\le \sum_{Q\in\mathcal{W}_0} \mathbf{1}_Q(x)\bigl(M^{\nu}_Q u(x)\bigr)^{p-\varepsilon}
\]
for every $x\in B_0$. Hence,
\begin{align*}
\int_{B_0}  \bigl(M^{\nu}_{\textup{loc}} u(x)\bigr)^{p-\varepsilon}\,d\mu(x)
\le \sum_{Q\in\mathcal{W}_0} \int_{Q} \bigl(M^{\nu}_{Q} u(x)\bigr)^{p-\varepsilon}\,d\mu(x)\,.
\end{align*}
At this stage, we fix a ball $Q\in\mathcal{W}_0$ and write  
the corresponding integral as follows:
\begin{align*}
\int_{Q} \bigl(M^{\nu}_{Q} u(x)\bigr)^{p-\varepsilon}\,d\mu(x) =
(p-\varepsilon)\int_0^\infty \lambda^{p-\varepsilon} \mu(Q^\lambda)\,\frac{d\lambda}{\lambda}\,.
\end{align*}
Since $Q^\lambda=Q=Q^{2\lambda}$ for every $\lambda \in (0,\lambda_Q/2)$, 
we find that
\begin{align*}
(p-\varepsilon)\int_0^{\lambda_Q/2} \lambda^{p-\varepsilon} \mu(Q^\lambda)\,\frac{d\lambda}{\lambda}&=\frac{(p-\varepsilon)}{2^{p-\varepsilon}}\int_0^{\lambda_Q/2} (2\lambda)^{p-\varepsilon} \mu(Q^{2\lambda})\,\frac{d\lambda}{\lambda}\\
&\le \frac{(p-\varepsilon)}{2^{p-\varepsilon}}\int_0^{\infty} \sigma^{p-\varepsilon} \mu(Q^{\sigma})\,\frac{d\sigma}{\sigma}
\\&=\frac{1}{2^{p-\varepsilon}}\int_Q \bigl(M^{\nu}_{Q} u(x)\bigr)^{p-\varepsilon}\,d\mu(x)\,.
\end{align*}
On the other hand, by Lemma \ref{l.mainl_local}, for each $\lambda>\lambda_Q/2$,
\begin{align*}
\lambda^{p-\varepsilon} \mu(Q^\lambda)
\le C_2\lambda^{-\varepsilon}\biggl[(2^{-k\alpha} & +\Psi(C_22^{-k\alpha}))(2^{k}\lambda)^p\mu(U^{2^k \lambda}_{Q^*})\\
& + \frac{K_{p,p}}{k^p} \sum_{j=k}^{2k-1}  (2^{j}\lambda)^p \mu(U^{2^j \lambda}_{Q^*})
+K_{p,p}\int_{U^{\lambda}_{Q^*}\setminus U^{4^k\lambda}} g^p\,d\mu\,\biggr]\,,
\end{align*}
where  $C_2=C(\delta,p,c_\mu,c_\nu,c_{\nu,\mu},\lVert \Psi\rVert_\infty)>0$.
Since $p-\varepsilon>1$, it follows that 
\begin{align*}
\int_Q \bigl( M^{\nu}_{Q} u(x)\bigr)^{p-\varepsilon}\,d\mu(x)
&\le 2(p-\varepsilon)\int_{\lambda_Q/2}^\infty \lambda^{p-\varepsilon} \mu(Q^\lambda)\,\frac{d\lambda}{\lambda}\\
&\le 2pC_2\bigl(I_1(Q) + I_2(Q) + I_3(Q)\bigr)\,,
\end{align*}
where
\begin{align*}
I_1(Q)&=(2^{-k\alpha}+\Psi(C_22^{-k\alpha}))2^{k\varepsilon}\int_{0}^\infty (2^{k}\lambda)^{p-\varepsilon} \mu(U^{2^k \lambda}_{Q^*})  \,\frac{d\lambda}{\lambda}\,,\qquad \\
I_2(Q)&= \frac{K_{p,p}}{k^p}\sum_{j=k}^{2k-1} 2^{j\varepsilon}\int_0^\infty (2^j \lambda)^{p-\varepsilon} \mu(U^{2^j \lambda}_{Q^*})\,\frac{d\lambda}{\lambda}\,, \\
I_3(Q)&= K_{p,p}
\int_0^\infty \lambda^{-\varepsilon} \int_{U^{\lambda}_{Q^*}\setminus U^{4^k\lambda}} g(x)^p\,d\mu(x)\,\frac{d\lambda}{\lambda}\,.
\end{align*}
By (W2) we have $\sum_{Q\in\mathcal{W}_0} \mathbf{1}_{Q^*}\le C(c_\nu)\mathbf{1}_{B_0}$. Hence, we can now continue to estimate as follows. First,
\begin{align*}
\sum_{Q\in\mathcal{W}_0} I_1(Q)&\le C(c_\nu)
(2^{-k\alpha}+\Psi(C_22^{-k\alpha}))2^{k\varepsilon}\int_0^\infty (2^k \lambda)^{p-\varepsilon}\mu(U^{2^k\lambda})\frac{d\lambda}{\lambda}\\
&\le \frac{C(c_\nu)}{p-\varepsilon}(2^{-k\alpha}+\Psi(C_22^{-k\alpha}))2^{k\varepsilon}\int_{B_0} \bigl(M^{\nu} u(x)\bigr)^{p-\varepsilon}\,d\mu(x)\,.
\end{align*}
Here the last upper bound is of the required form, since
$p-\varepsilon>1$.
Second, 
\begin{align*}
\sum_{Q\in\mathcal{W}_0} I_2(Q)&
\le C(c_\nu)\frac{K_{p,p}}{k^p}\sum_{j=k}^{2k-1}  2^{j\varepsilon}\int_0^\infty (2^j \lambda)^{p-\varepsilon} \mu(U^{2^j \lambda})\,\frac{d\lambda}{\lambda}\\
&\le \frac{C(c_\nu)K_{p,p}}{k^p(p-\varepsilon)}\Biggl(\sum_{j=k}^{2k-1} 2^{j\varepsilon}\Biggr)\int_{B_0} \bigl(M^{\nu} u(x)\bigr)^{p-\varepsilon}\,d\mu
\\&\le C(c_\nu) \frac{K_{p,p}4^{k\varepsilon}}{k^{p-1}}\int_{B_0} \bigl(M^{\nu} u(x)\bigr)^{p-\varepsilon}\,d\mu\,.
\end{align*}
Third, by using also Fubini's theorem,
\begin{align*}
\sum_{Q\in\mathcal{W}_0} I_3(Q)&
\le C(c_\nu)K_{p,p}\int_{B_0\setminus \{M^{\nu} u=0\}} \biggl(   \int_0^\infty \lambda^{-\varepsilon}  
\mathbf{1}_{U^{\lambda}\setminus U^{4^k\lambda}}(x)  \frac{d\lambda}{\lambda}  \biggr)g(x)^p\,d\mu(x)\\
& \le  C(c_\nu)C(k,\varepsilon) K_{p,p}\int_{B_0\setminus \{M^{\nu} u=0\}}g(x)^p (M^{\nu} u(x))^{-\varepsilon}\,d\mu(x)\,.\end{align*}
Combining the estimates above, we arrive at the desired inequality
\eqref{e.loc_des}.

This concludes the proof of Theorem \ref{t.main_local}, and hence we have now also completed the proof of Theorem~\ref{t.main_thm}.
\qed

\section{Modified Lipschitz extension}\label{s.Whitney_ext}

We now turn to the proof of Theorem~\ref{t.global_epsilon_0}, related to the self-improvement of maximal
Poincar\'e inequalities. The most important implication in Theorem~\ref{t.global_epsilon_0}, (C) $\Longrightarrow$ (D),
is in fact a consequence of the more general Theorem~\ref{t.global_improvement}, which we state and prove in the
following Section~\ref{s.global}. However, in the proof of the latter result we need to be able to extend Lipschitz
functions in such a way that the extension does not decrease the global sharp maximal function too much,
and thus we cannot immediately apply the McShane extension~\eqref{McShane}.
The purpose of this rather independent section is to prove the technical Theorem~\ref{t.monotone_extension} 
by constructing one example of a suitable extension. 
Nevertheless, we emphasize that the proof of Theorem~\ref{t.global_improvement} works for any $C\lambda$-Lipschitz extension satisfying the crucial 
estimate~\eqref{e.curious_ineq} below.

Fix $1\le p<\infty$ and let $M^{\nu,p}u=M^{\nu,p}_{\mathcal{X}}u$ be the
maximal function of $u\in \Lip(X)$ with respect to  
the family
\[\mathcal{X}=\{B\subset X\,:\, B\text{ is a ball}\} \]
of all balls in $X$, as in~\eqref{eq.global_balls}.

\begin{theorem}\label{t.monotone_extension}
Let $1\le p<\infty$. Assume that $X=(X,d,\nu,\mu)$ is a geodesic two-measure space
and that $u\in \Lip(X)$ has a bounded support.  Let $\lambda>0$ and denote
\begin{equation}\label{e.level_set}
U_\lambda = \{x\in X\,:\, M^{\nu,p} u(x)>\lambda\}\,.
\end{equation}
Then there exists a $C\lambda$-Lipschitz function $u_\lambda\colon X\to \R$ such that
$u|_{X\setminus U_\lambda}=u_\lambda|_{X\setminus U_\lambda}$ and 
\begin{equation}\label{e.curious_ineq}
M^{\nu,p} u(x)\le CM^{\nu,p}u_\lambda(x)
\end{equation}
for all $x\in X\setminus U_\lambda$. Here the constant $C=C(c_\nu,p)>0$
is independent of $u$ and $\lambda$.
\end{theorem}

The rest 
of this section is devoted to the proof of
Theorem \ref{t.monotone_extension}. In Section \ref{s.whitney_first}
we first construct a standard Whitney type Lipschitz extension $v_\lambda$ of $u|_{X\setminus U_\lambda}$, which 
we 
then modify in Section \ref{s.modification} in order to ensure
that inequality \eqref{e.curious_ineq} is valid.
At the end of Section \ref{s.modification}  we collect all of the required estimates and provide a proof for  
Theorem \ref{t.monotone_extension}.
Most of Section \ref{s.whitney_first} is standard, but  
we provide several details  --- both for the convenience of the reader 
and as a background for Section \ref{s.modification}.

At this stage, we fix a function $u\in \Lip(X)$ with a bounded support, a number $\lambda>0$, and the corresponding level set $U_\lambda$ 
as in Theorem \ref{t.monotone_extension}. 
Since the function $u$ has a bounded support, it is straightforward to show that
$U_\lambda$ is a bounded set.
We will  also  need the fact
that there is a constant $ C(c_\nu)>0$ for which
\begin{equation}\label{e.msharp}
\lvert u(x)-u(y)\rvert\le C(c_\nu)d(x,y)(M^{\nu,p} u(x)+ M^{\nu,p}u(y))
\end{equation}
whenever $x,y\in X$.
The proof of inequality \eqref{e.msharp} is based on a standard `telescoping' argument; cf.\ 
\cite[Lemma 4.12]{KinnunenLehrbackVahakangasZhong}.
As a consequence of inequality \eqref{e.msharp}, in the sequel we are allowed to assume that
$\emptyset\not=U_\lambda\not=X$.

\subsection{Standard Whitney extension}\label{s.whitney_first}

Since $\emptyset\not=U_\lambda\not=X$ is an open set in $X$, it admits a Whitney ball cover; cf.\ Section \ref{ss.Whitney}.   More specifically,
there is a   countable   family
$\mathcal{W}_\lambda=\mathcal{W}_\lambda(U_\lambda)$
of Whitney balls $Q\in\mathcal{W}_\lambda$.
These balls are of the form $Q=B(x_Q,r_Q)\in\mathcal{W}_\lambda$,
with
$x_Q\in U_\lambda$ and 
\[
r_Q=\frac{\dist(x_Q, X\setminus U_\lambda)}{128}>0\,.
\]
The $4$-dilated Whitney ball is denoted 
again  by $Q^*=4Q=B(x_Q,4 r_Q)$ whenever $Q\in\mathcal{W}_\lambda$,
 and as in Section \ref{ss.Whitney}, the Whitney balls have 
the following covering properties with bounded overlap: \begin{equation}\label{e.w_u_l}
U_\lambda=\bigcup_{Q\in\mathcal{W}_\lambda} Q\,,\qquad 
\sum_{Q\in\mathcal{W}_\lambda} \mathbf{1}_{Q^*}\le C(c_\nu)\mathbf{1}_{U_\lambda}\,.
\end{equation}
Moreover, the Whitney ball construction
can be done in such a way that $\{\frac{1}{10}Q\,:\,Q\in\mathcal{W}_\lambda\}$ is a pairwise disjoint family.

Next we construct a partition of unity with respect to the above  Whitney balls. Fix 
$Q\in\mathcal{W}_\lambda$, and define a function $\hat \psi_Q\colon X\to [0,1]$
 by setting for each $x\in X$  that 
\begin{align*}
\hat{\psi}_Q(x)&=\min\left\{1,\max\left\{0,2-\frac{d(x,x_Q)}{r_Q}\right\}\right\}
\\&=\begin{cases}
1\,,\qquad & d(x,x_Q)< r_Q\,;\\
2-\frac{d(x,x_Q)}{r_Q}\,,\qquad & r_Q \le d(x,x_Q)\le 2r_Q\,;\\
0\,,\qquad & d(x,x_Q)>2r_Q\,.
\end{cases}
\end{align*}
Observe that $\hat \psi_Q$ is $(1/r_Q)$-Lipschitz, $\hat \psi_Q=1$ in $Q$, 
and $\hat \psi_Q=0$ in $X\setminus Q^*$.
We then define a function $\psi_Q\colon U_\lambda\to [0,1]$  by
\begin{equation}\label{e.psi_def}
\psi_Q(x)=\frac{\hat \psi_Q(x)}{\sum_{P\in\mathcal{W}_\lambda} \hat \psi_P(x)}\,,\qquad x\in U_\lambda\,.
\end{equation}
The function $\psi_Q$ is well-defined, 
 since, by~\eqref{e.w_u_l}, for each $x\in U_\lambda$ 
there is $P_x\in\mathcal{W}_\lambda$ such that $x\in P_x$. 
Thus  $\hat \psi_{P_x}(x)=1$,  and so we see that 
\begin{equation}\label{e.bmass}
1\le \sum_{P\in\mathcal{W}_\lambda} \hat \psi_P(x)\le \sum_{P\in\mathcal{W}_\lambda} \mathbf{1}_{P^*}(x)\le C(c_\nu)\,,\qquad x\in U_\lambda\,.
\end{equation}

\begin{lemma}\label{l.Q-Lipschitz}  
There is a  constant $C=C(c_\nu)>0$ such that the function $\psi_Q\colon U_\lambda\to [0,1]$ 
is $(C/r_Q)$-Lipschitz in $U_\lambda$ whenever $Q\in\mathcal{W}_\lambda$.  
\end{lemma}

\begin{proof}
Fix $Q\in\mathcal{W}_\lambda$
and  $x,y\in U_\lambda$.
Let us first consider the case $x,y\in Q^*$.
  If $P\in\mathcal{W}_\lambda$ and $\{x,y\}\cap P^*\not=\emptyset$,  then it is straightforward to check that   $r_P/2\le r_Q\le 2r_P$.
Hence, by  \eqref{e.w_u_l} and inequality 
$\mathbf{1}_{U_\lambda}\le \sum_{P\in\mathcal{W}_\lambda}\hat\psi_P$, we obtain that
\begin{align*}
\lvert \psi_Q(x)-\psi_Q(y)\rvert&= \left\lvert 
\frac{\hat \psi_Q(x)\sum_{P\in\mathcal{W}_\lambda}\hat \psi_P(y)-\hat \psi_Q(y)\sum_{P\in\mathcal{W}_\lambda}\hat \psi_P(x)}{\sum_{P\in\mathcal{W}_\lambda}\hat\psi_P(y)\sum_{P\in\mathcal{W}_\lambda}\hat \psi_P(x)}\right\rvert\\
&\le \left\lvert \hat \psi_Q(x)\sum_{P\in\mathcal{W}_\lambda}\hat \psi_P(y)-\hat \psi_Q(y)\sum_{P\in\mathcal{W}_\lambda}\hat \psi_P(x)\right\rvert\\
&=\left\lvert\sum_{P\in\mathcal{W}_\lambda} \left( \hat\psi_Q(x)\bigl(\hat\psi_P(y)-\hat\psi_P(x)\bigr)+ \hat\psi_P(x)\bigl(\hat\psi_Q(x)-\hat\psi_Q(y)\bigl)\right)\right\rvert\\
&\le \sum_{P\in\mathcal{W}_\lambda}
\Bigr( \hat\psi_Q(x)\left\lvert \hat\psi_P(y)-\hat\psi_P(x)\big\rvert+ \hat\psi_P(x)\big\lvert\hat\psi_Q(x)-\hat\psi_Q(y)\right\rvert\Bigr)\\
&\le \sum_{\substack{P\in\mathcal{W}_\lambda\\\{x,y\}\cap P^*\not=\emptyset}}
\left( \frac{\hat\psi_Q(x)}{r_P}+ \frac{\hat\psi_P(x)}{r_Q}\right)d(x,y) 
\le \frac{C(c_\nu)}{r_Q} d(x,y)\,.
\end{align*}
On the other hand, if $x\in Q^*$ and $y\not\in Q^*$, then
\[
\lvert \psi_Q(x)-\psi_Q(y)\rvert 
=\psi_Q(x)\le \hat \psi_Q(x)=\lvert \hat \psi_Q(x)-\hat\psi_Q(y)\rvert\le \frac{d(x,y)}{r_Q}\,.
\]
The case $x\not\in Q^*$ and $y\in Q^*$ is treated in a similar way.
Finally, if $x,y\not\in Q^*$, then the Lipschitz condition
is trivially valid since $\lvert \psi_Q(x)-\psi_Q(y)\rvert=0$.
\end{proof}

The preliminary extension of $u$ is now
defined to be the function $v_\lambda\colon X\to \R$,
\begin{equation}\label{e.vl_def}
v_\lambda(x)=\begin{cases}
\sum_{Q\in\mathcal{W}_\lambda}u_{Q;\nu}\psi_Q(x)\,,\qquad &x\in U_\lambda\,,\\
u(x)\,,\qquad &x\in X\setminus U_\lambda\,;
\end{cases}
\end{equation}
here we  recall that $u_{Q;\nu}=\vint_Q u(y)\,d\nu(y)$ for every $Q\in\mathcal{W}_\lambda$.

By the above  definition and  the  inequality in \eqref{e.w_u_l}, it is clear that $v_\lambda\colon X\to\R$ is a well defined extension 
of $u|_{X\setminus U_\lambda}$. In  the following lemma
 we show that this extension is a Lipschitz function.

\begin{lemma}\label{l.v_Lip}
The function $v_\lambda\colon X\to \R$
is $C\lambda$-Lipschitz in $X$ with $C=C(c_\nu)>0$. 
\end{lemma}

\begin{proof}
If $x,y\in X\setminus U_\lambda$, then by inequality   \eqref{e.msharp} we have
\begin{align*}
\lvert v_\lambda(x)-v_\lambda(y)\rvert&=\lvert u(x)-u(y)\rvert
\\&\le C(c_\nu)d(x,y)\bigl(M^{\nu,p}u(x)+M^{\nu,p} u(y)\bigr)
\le C(c_\nu) \lambda d(x,y)\,. 
\end{align*}

Then we consider the case $x\in U_\lambda$
and $y\in X\setminus U_\lambda$.
Assume first that 
$d(x,y)\le 2\dist(x,X\setminus U_\lambda)$, and let 
$Q\in\mathcal{W}_\lambda$ be any Whitney ball such that  
$\psi_Q(x)\not=0$. Then 
$x\in Q^*$, and so 
\begin{align*}
d(x,y)&\le 2\dist(x,X\setminus U_\lambda) 
\le 2\bigl(d(x,x_Q)+\dist(x_Q,X\setminus U_\lambda)\bigr)\\
&\le 2(4r_Q+128 r_Q)=264r_Q\,.
\end{align*}
On the other hand, since $d(x_Q,x)<4r_Q$,
\[
d(x,y)\ge \dist(x,X\setminus U_\lambda)\ge 
\dist(x_Q,X\setminus U_\lambda)-d(x_Q,x)>128r_Q-4r_Q>r_Q\,.
\]
That is, we have
\[
r_Q< d(x,y)< 264r_Q\,.
\]
Since $x\in Q^*$ and $d(x,y)<264r_Q$, we find that $y\in 268Q$. 
Arguing as in \cite[Lemma 4.12]{KinnunenLehrbackVahakangasZhong}, 
we obtain that  
\[
\lvert u(y)-u_{Q;\nu}\rvert\le \lvert u(y)-u_{268Q;\nu}\rvert + \lvert u_{268Q;\nu}-u_{Q;\nu}\rvert\le C(c_\nu)r_Q M^{\nu,p}u(y)
\le C(c_\nu) \lambda d(x,y)
\]
whenever $Q\in\mathcal{W}_\lambda$ is such that $\psi_Q(x)\not=0$.  
Hence,
\begin{align*}
\lvert v_\lambda(x)-v_\lambda(y)\rvert
&= \bigg\lvert \sum_{Q\in\mathcal{W}_\lambda} u_{Q;\nu}\psi_Q(x)-u(y)\bigg\rvert
= \bigg\lvert \sum_{Q\in\mathcal{W}_\lambda} \psi_Q(x)\bigl(u_{Q;\nu}-u(y)\bigr)\bigg\rvert \\
&\le \sum_{Q\in\mathcal{W}_\lambda} \psi_Q(x)\lvert u_{Q;\nu}-u(y)\rvert\le C(c_\nu)\lambda d(x,y)\,.
\end{align*}
This concludes the proof of the lemma in the case when $x\in U_\lambda$ and $y\in X\setminus U_\lambda$ are such that $d(x,y)\le 2\dist(x,X\setminus U_\lambda)$.

Next we treat the case when $x\in U_\lambda$
and $y\in X\setminus U_\lambda$ satisfy
$d(x,y)> 2\dist(x,X\setminus U_\lambda)$.
Now there exists a point $z\in X\setminus U_\lambda$
such that
$d(x,z)\le 2\dist(x,X\setminus U_\lambda)$.
Since $d(x,z)<d(x,y)$ and $d(z,y)\le d(x,z)+d(x,y)<2d(x,y)$,
we can resort to  the already established cases as follows: 
\begin{align*}
\lvert v_\lambda(x)-v_\lambda(y)\rvert
&\le \lvert v_\lambda(x)-v_\lambda(z)\rvert
+\lvert v_\lambda(z)-v_\lambda(y)\rvert
\\&\le C(c_\nu)\lambda d(x,z)+C(c_\nu)\lambda d(z,y)
\le C(c_\nu)\lambda d(x,y)\,.
\end{align*}

The case $x\in X\setminus U_\lambda$ and $y\in U_\lambda$ 
follows from the preceding arguments via symmetry.

It remains to consider the case $x,y\in U_\lambda$.
By symmetry, we may clearly assume
that 
\begin{equation}\label{e.symm1}
\dist(y,X\setminus U_\lambda)\le \dist(x,X\setminus U_\lambda)\,.
\end{equation}
Denote $d=\dist(x,X\setminus U_\lambda)>0$.  
Now, if either one of the following two inequalities 
\begin{equation}\label{e.additional}
\dist(y,X\setminus U_\lambda)\ge d/5\qquad\text{ and }\qquad d(x,y)\le d
\end{equation}
fails, then we are essentially done. Namely,  
if $\dist(y,X\setminus U_\lambda)<d/5$, then
\[
d(x,y)\ge \dist(x,X\setminus U_\lambda)-\dist(y,X\setminus U_\lambda)> 4d/5\,.
\]
Moreover, then there exists a point $z\in X\setminus U_\lambda$
such that $d(y,z)\le 2d/5$, and thus 
\begin{equation}\label{e.opt1}
d(x,z)+d(z,y)\le d(x,y)+2d(y,z)\le d(x,y)+4d/5\le 2d(x,y)\,.
\end{equation}
On the other hand, if the second inequality in \eqref{e.additional} fails, that is, if $d(x,y)>d$, then we first choose 
a point $z\in X\setminus U_\lambda$ such that
$d(x,z)<2d$. In this case we then have
\begin{equation}\label{e.opt2}
d(x,z)+d(z,y)\le 2d(x,z)+d(x,y)< 4d+d(x,y)< 5d(x,y)\,.
\end{equation}
Having either inequality \eqref{e.opt1} or \eqref{e.opt2} for some $z\in X\setminus U_\lambda$,  
we can always resort to the previously established
cases in order to see that
\begin{align*}
\lvert v_\lambda(x)-v_\lambda(y)\rvert
&\le \lvert v_\lambda(x)-v_\lambda(z)\rvert + \lvert v_\lambda(z)-v_\lambda(y)\rvert \\&\le C(c_\nu)\lambda (d(x,z)+d(z,y))\le C(c_\nu)\lambda d(x,y)\,.
\end{align*}
Therefore, it suffices to establish the Lipschitz property of $v_\lambda$
when both inequalities
in \eqref{e.additional} are  valid. 

Denote $B=B(x,6d)$.
Fix $Q\in\mathcal{W}_\lambda$
such that $\{x,y\}\cap Q^*\not=\emptyset$ and let $z\in \{x,y\}\cap Q^*$. Then, by \eqref{e.symm1} and \eqref{e.additional}, we have
\[
d/5\le \dist(z,X\setminus U_\lambda)\le d\quad \text{ and } \quad r_Q\le 
\dist(z,X\setminus U_\lambda)<132r_Q\,.
\]
In particular, by Lemma \ref{l.Q-Lipschitz}, the function
$\psi_Q$ is $(C/d)$-Lipschitz with $C=C(c_\nu)$. 
By the second inequality in  \eqref{e.additional}
we also have that  $Q\subset B$.
Fix a point $w\in B\setminus U_\lambda$. Then 
\begin{align*}
\lvert u_{Q;\nu}-u_{B;\nu}\rvert&\le \vint_{Q}\lvert u-u_{B;\nu}\rvert\,d\nu
\le C(c_\nu)\vint_B\lvert u-u_{B;\nu}\rvert\,d\nu \\
&\le C(c_\nu)\,d\,\biggl(\frac{1}{\diam(B)^p}\vint_B\lvert u-u_{B;\nu}\rvert^p\,d\nu\biggr)^{1/p}\\
&\le C(c_\nu) dM^{\nu,p}u(w)\le C(c_\nu)\lambda d\,.
\end{align*}

By  the previous estimates and \eqref {e.w_u_l}, we can now proceed as follows: 
\begin{align*}
\lvert v_\lambda(x)-v_\lambda(y)\rvert
&=\bigg\lvert \sum_{Q\in\mathcal{W}_\lambda} u_{Q;\nu}\psi_Q(x)-\sum_{Q\in\mathcal{W}_\lambda} u_{Q;\nu}\psi_Q(y)\bigg\rvert\\
&=\bigg\lvert \sum_{Q\in\mathcal{W}_\lambda} \bigl(\psi_Q(x)-\psi_Q(y)\bigr)(u_{Q;\nu}-u_{B;\nu})\bigg\rvert\\
&\le \sum_{\substack{Q\in\mathcal{W}_\lambda\\\{x,y\}\cap Q^*\not=\emptyset}} \lvert\psi_Q(x)-\psi_Q(y)\rvert \cdot \lvert u_{Q;\nu}-u_{B;\nu}\rvert\\
&\le  C(c_\nu)\cdot \frac{C(c_\nu)}{d}d(x,y)\cdot C(c_\nu) \lambda d=C(c_\nu)\lambda d(x,y)\,.
\end{align*}
This concludes our treatment of the last case $x,y\in U_\lambda$. \end{proof}

\subsection{Modification of the extension}\label{s.modification}

In Section \ref{s.whitney_first} we constructed a $C\lambda$-Lipschitz
extension $v_\lambda\colon X\to \R$ of $u|_{X\setminus U_\lambda}$.
We still need to modify this extension in the set $U_\lambda$ 
in order to ensure that  inequality \eqref{e.curious_ineq} becomes valid.
This modification has to be done carefully, since we want
the Lipschitz constant of the modified function $u_\lambda$ to be in control.
Therefore, in some of the following arguments, we need to track
the Lipschitz constants more quantitatively.

For this purpose we introduce a new
notation $\kappa=C\lambda$, where $C=C(c_\nu)>0$ is the constant from Lemma \ref{l.v_Lip}.
In particular, we then have
\begin{equation}\label{e.kappa_Lip}
\lvert v_\lambda(x)-v_\lambda(y)\rvert\le \kappa d(x,y)\,,\qquad x,y\in X\,.
\end{equation}
Whenever such a  quantitative tracking is not necessary, we resort to the usual notation $C\lambda$.

The number $\kappa>0$ appears in the properties of the following 
bump functions.
For each $Q\in\mathcal{W}_\lambda$, we let
$b_Q\colon X\to \R$ be a function satisfying the following conditions (B1)--(B5):
\begin{itemize}
\item[(B1)] $S_Q=\{x\in X\,:\, b_Q(x)\not=0\}\subset \frac{1}{10}Q$\,;
\item[(B2)] $\nu(S_Q)\le 2^{-1}\nu(\frac{1}{10}Q)$\,;
\item[(B3)] $b_Q$ is $C(c_\nu,p)\lambda$-Lipschitz   in $X$\,;  
\item[(B4)] $\int_{Q} b_Q\,d\nu=0$\,;
\item[(B5)] $\int_Q \lvert b_Q\rvert^p\,d\nu \ge 2^p\kappa^p r_Q^p \nu(Q)$\,.
\end{itemize}

The actual construction of $b_Q$ relies on the observation that
if $B=B(x,r)\subsetneq X$ is a ball, then there
exists two disjoint balls $B_1=B(x_1,r/4)\subset B$ and $B_2=B(x_2,r/4)\subset B$.
This  follows from the fact that $X$ is a geodesic two-measure space.
Now, the above observation with $B=\frac{1}{10}Q$ gives us a ball 
$Q'\subset \frac{1}{10}Q$ of radius $r_Q/40$ such that
$\nu(Q')\le 2^{-1}\nu(\frac{1}{10}Q)$.
By applying the above observation again, but this time with $B=Q'$, we
choose 
two disjoint balls $Q'_1=B(x_1,r_Q/160)\subset Q'$ and $Q'_2=B(x_2,r_Q/160)\subset Q'$.
For each $x\in X$, we let
\[
\varphi_{Q,j}(x)=\max\{0,r_Q-160\cdot d(x,x_j)\}\,,\qquad j\in \{1,2\}\,.
\]
With the aid of $160$-Lipschitz functions $\varphi_{Q,1}$ and $\varphi_{Q,2}$, both of
which are zero in $X\setminus Q'\supset X\setminus \frac{1}{10}Q$, we then define
a $C(c_\nu)$-Lipschitz function $\varphi_Q\colon X\to \R$ by setting
\[
\varphi_Q=\varphi_{Q,1} - \left(\frac{\int_X \varphi_{Q,1}\,d\nu}{\int_X \varphi_{Q,2}\,d\nu}\right)\varphi_{Q,2}\,.
\]
Observe that $\int_Q \varphi_Q\,d\nu=\int_X \varphi_Q\,d\nu=0$.
In order to ensure that condition (B5) holds, we  need to normalize the above function $\varphi_Q$ by defining
\[
b_Q = \frac{2\kappa}{C(c_\nu,p)^{1/p}}\varphi_Q\,,
\]
where the constant $C(c_{\nu},p)>0$ is
such that inequality
 $\int_X \lvert \varphi_Q\rvert^p\,d\nu\ge C(c_\nu,p)r_Q^p\nu(Q)$ holds.
It is a tedious but straightforward task to check that the function $b_Q\colon X\to \R$ above satisfies
all of the required properties (B1)--(B5). We leave further details to the interested reader.

We now define a function
\begin{equation}\label{e.b_def}
b=\sum_{Q\in\mathcal{W}_\lambda} b_Q\colon X\to \R
\end{equation}
that will be used to modify $v_\lambda$ in the set $U_\lambda$. 
 First, we record some 
basic properties of $b$.

\begin{lemma}\label{l.b_Lip}
The function $b$ is well defined in $X$ and $b=0$ in $X\setminus U_\lambda$.
Moreover, the function $b$ is $C\lambda$-Lipschitz in $X$ with $C=C(c_\nu,p)>0$.
\end{lemma}

\begin{proof}
The condition (B1) above and the second property in \eqref{e.w_u_l}
are used several times below.
First of all, they imply that $b$ is well defined in $X$.
Moreover, 
\[
 b =  \sum_{Q\in\mathcal{W}_\lambda} b_Q
 = \sum_{Q\in\mathcal{W}_\lambda} 
 \mathbf{1}_{U_\lambda}b_Q=\mathbf{1}_{U_\lambda}\sum_{Q\in\mathcal{W}_\lambda}  b_Q=\mathbf{1}_{U_\lambda} b\,.
\]
Hence $b=0$ in $X\setminus U_\lambda$.
Finally, if $x,y\in X$, then
\begin{align*}
\lvert b(x)-b(y)\rvert &= 
\bigg\lvert \sum_{Q\in\mathcal{W}_\lambda} b_Q(x)-\sum_{Q\in\mathcal{W}_\lambda} b_Q(y)\bigg\rvert\\
&\le  \sum_{\substack{Q\in\mathcal{W}_\lambda\\\{x,y\}\cap Q\not=\emptyset}} \lvert b_Q(x)- b_Q(y)\rvert
\le C(c_\nu)\cdot C(c_\nu,p)\lambda d(x,y)\,,
\end{align*}
 and so $b$ is Lipschitz in $X$ with a constant $C(c_\nu,p)\lambda>0$. 
\end{proof}

Finally, we define 
\begin{equation}\label{e.curious_ext}
u_\lambda\colon X\to \R,\ u_\lambda = v_\lambda + b\,, 
\end{equation}
where  $v_\lambda$ is defined by \eqref{e.vl_def} and $b$ is defined by \eqref{e.b_def}.
A combination of Lemma~\ref{l.v_Lip} and Lemma~\ref{l.b_Lip} shows that $u_\lambda\colon X\to \R$
is an $C(c_\nu,p)\lambda$-Lipschitz extension of $u|_{X\setminus U_\lambda}$.
It remains to show that inequality \eqref{e.curious_ineq} holds for $u_\lambda$. We
begin this task by formulating and proving the following quantitative lemma that 
 is an auxiliary result for Lemma \ref{l.monotone_extension}.

\begin{lemma}\label{l.aux}
Let 
$Q\in\mathcal{W}_\lambda$. Then 
\[
\kappa^p r_Q^p \nu(Q)\le  \int_{\frac{1}{10}Q} \lvert u_\lambda-(u_\lambda)_{\frac{1}{10}Q;\nu}\rvert^p\,d\nu\,.
\]
\end{lemma}

\begin{proof}
  We remark that quantitative tracking of Lipschitz and other constants is needed below.  
Fix a Whitney ball $Q\in\mathcal{W}_\lambda$.
By inequality \eqref{e.kappa_Lip},
\begin{equation}\label{e.error}
\begin{split}
\biggl(\int_{\frac{1}{10}Q} \lvert v_\lambda -(v_\lambda)_{\frac{1}{10}Q;\nu}\rvert^p\,d\nu\biggr)^{1/p} 
& \le \biggl(\int_{\frac{1}{10}Q}\vint_{\frac{1}{10}Q}\lvert v_\lambda(x)-v_\lambda(y)\rvert^p\,d\nu(y)\,d\nu(x)\biggr)^{1/p} \\
& \le \kappa r_Q\nu(Q)^{1/p}\,.
\end{split}
\end{equation}
 From estimate \eqref{e.error} and conditions (B1) and (B5) it then follows that
\begin{align*}
\kappa r_Q\nu(Q)^{1/p}&=2\kappa r_Q \nu(Q)^{1/p}-\kappa r_Q\nu(Q)^{1/p}
\\&\le \biggl(\int_{\frac{1}{10}Q} \lvert b_Q\rvert^p\,d\nu\biggr)^{1/p} - \biggl(\int_{\frac{1}{10}Q} \lvert v_\lambda -(v_\lambda)_{\frac{1}{10}Q;\nu}\rvert^p\,d\nu\biggr)^{1/p}\\
&\le \biggl(\int_{\frac{1}{10}Q} \lvert b_Q+v_\lambda -(v_\lambda)_{\frac{1}{10}Q;\nu}\rvert^p\,d\nu\biggr)^{1/p}\,.
\end{align*}
Recall that the family $\{\frac{1}{10} P\,:\, P\in\mathcal{W}_\lambda\}$ is pairwise
disjoint. Hence, by condition (B1), we find that
$u_\lambda=v_\lambda+b_Q$ in $\frac{1}{10}Q$. This fact together with (B1) and (B4) implies that
\[
\mathbf{1}_{\frac{1}{10}Q}(b_Q+v_\lambda -(v_\lambda)_{\frac{1}{10}Q;\nu})=
\mathbf{1}_{\frac{1}{10}Q}(u_\lambda-(u_\lambda)_{\frac{1}{10}Q;\nu})\,.
\]
By concluding from above, we obtain
\[
\kappa r_Q\nu(Q)^{1/p}\le \biggl(\int_{\frac{1}{10}Q} \lvert u_\lambda-(u_\lambda)_{\frac{1}{10}Q;\nu}\rvert^p\,d\nu\biggr)^{1/p}\,,
\]
 and the claim follows by raising 
both sides to power $p$. 
\end{proof}

The following lemma ensures that inequality \eqref{e.curious_ineq} holds for 
the modified extension $u_\lambda$.

\begin{lemma}\label{l.monotone_extension}
There is a constant $C=C(c_\nu,p)>0$ such that 
\[M^{\nu,p} u(x)\le CM^{\nu,p}u_\lambda(x)\] for all $x\in X\setminus U_\lambda$. 
\end{lemma}

\begin{proof}
Fix a  ball $B\subset X$
such that $B\setminus U_\lambda\not=\emptyset$. 
By inequalities \eqref{e.doubling} and \eqref{e.diams}, it suffices to prove that
\[
\int_B \lvert u-u_{B;\nu}\rvert^p\,d\nu\le  C(c_\nu,p) \int_{2B}\lvert u_\lambda - (u_\lambda)_{2B;\nu}\rvert^p\,d\nu\,.
\]
Since $\int_B\lvert u-u_{B;\nu}\rvert^p\,d\nu\le 2^{p}\int_B \lvert u-(u_\lambda)_{2B;\nu}\rvert^p\,d\nu$,
it is enough to show that
\begin{equation}\label{e.suffices}
\int_B \lvert u-(u_\lambda)_{2B;\nu}\rvert^p\,d\nu\le  C(c_\nu,p) \int_{2B} \lvert u_\lambda-(u_\lambda)_{2B;\nu}\rvert^p\,d\nu\,.
\end{equation}
  To this end, we first observe that
$u=v_\lambda=u_\lambda$ in $X\setminus U_\lambda$. Hence, 
\begin{equation}\label{e.account}
\int_{B}\lvert u-(u_\lambda)_{2B;\nu}\rvert^p\,d\nu
= \int_{B\setminus U_\lambda}\lvert u_\lambda-(u_\lambda)_{2B;\nu}\rvert^p\,d\nu+\int_{B\cap U_\lambda}\lvert u-(u_\lambda)_{2B;\nu}\rvert^p\,d\nu\,, 
\end{equation}
and therefore it suffices to estimate the integral over the set $B\cap U_\lambda$. 

If $Q\in\mathcal{W}_\lambda$, 
we denote $S_Q=\{x\in X\,:\, b_Q(x)\not=0\}$ and $R_Q=\frac{1}{10}Q\setminus S_Q$; recall (B1). 
Since the Whitney balls in $\mathcal{W}_\lambda$ cover the open set $U_\lambda$, we can estimate  
\begin{equation}\label{e.first_est}
\begin{split}
&\int_{B\cap U_\lambda} \lvert u-(u_\lambda)_{2B;\nu}\rvert^p\,d\nu
\le \sum_{\substack{Q\in\mathcal{W}_\lambda\\B\cap Q\not=\emptyset}}\int_Q \lvert u-(u_\lambda)_{2B;\nu}\rvert^p\,d\nu
\\&\le C(p)\sum_{\substack{Q\in\mathcal{W}_\lambda\\B\cap Q\not=\emptyset}}\left(\int_Q \lvert u-u_{Q;\nu}\rvert^p\,d\nu
+\nu(Q)\lvert u_{Q;\nu}-(u_\lambda)_{R_Q;\nu}\rvert^p+\nu(Q)\lvert (u_\lambda)_{R_Q;\nu}-(u_\lambda)_{2B;\nu}\rvert^p\right)\,.
\end{split}
\end{equation}
Fix $Q\in\mathcal{W}_\lambda$ such that $B\cap 
Q\not=\emptyset$. We claim that
\begin{equation}\label{e.tast}
\begin{split}
&\int_Q \lvert u-u_{Q;\nu}\rvert^p\,d\nu
+\nu(Q)\lvert u_{Q;\nu}-(u_\lambda)_{R_Q;\nu}\rvert^p+\nu(Q)\lvert (u_\lambda)_{R_Q;\nu}-(u_\lambda)_{2B;\nu}\rvert^p\\&\qquad \le C(c_\nu,p)\int_{\frac{1}{10}Q} \lvert u_\lambda-(u_\lambda)_{2B;\nu}\rvert^p\,d\nu\,.
\end{split}
\end{equation}

The last term on the left-hand side of \eqref{e.tast} can be estimated 
directly with condition (B2) as follows:  
\begin{align*}
\nu(Q)\lvert (u_\lambda)_{R_Q;\nu}-(u_\lambda)_{2B;\nu}\rvert^p
&\le \frac{\nu(Q)}{\nu(R_Q)}\int_{R_Q}\lvert u_\lambda-(u_\lambda)_{2B;\nu}\rvert^p\,d\nu \\
&\le C(c_\nu)\int_{\frac{1}{10}Q} \lvert u_\lambda-(u_\lambda)_{2B;\nu}\rvert^p\,d\nu\,.
\end{align*}

In the estimates for the other two terms on the left-hand side of \eqref{e.tast} 
we often rely on Lemma \ref{l.aux},
by which it suffices  to establish an upper bound of the form $C(c_\nu,p)\lambda^p r_Q^p\nu(Q)$;
recall that $\kappa=C(c_\nu)\lambda$. 
Indeed, the right-hand side of the estimate in Lemma \ref{l.aux} is bounded
from above by $2^p\int_{\frac{1}{10}Q} \lvert u_\lambda-(u_\lambda)_{2B;\nu}\rvert^p\,d\nu$.

As $Q\subset B(y_Q,256r_Q)$ for some $y_Q\in X\setminus 
U_\lambda$, we easily obtain the  desired upper bound for the integral term on the left-hand side of \eqref{e.tast}: 
\begin{align*}
\int_Q \lvert u-u_{Q;\nu}\rvert^p\,d\nu&\le 
2^{p} \int_Q\lvert u-u_{B(y_Q,256r_Q);\nu}\rvert^p\,d\nu\\
&\le 2^p
\int_{B(y_Q,256r_Q)} \lvert u-u_{B(y_Q,256r_Q);\nu}\rvert^p\,d\nu
\\&\le 
2^pM^{\nu,p}u(y_Q)^p(512r_Q)^p\nu(B(y_Q,256r_Q))
\le C(c_\nu,p)\lambda^p r_Q^p\nu(Q)\,.
\end{align*}

We still need to estimate the middle term on the left-hand side of \eqref{e.tast}.
In the sequel we will assume that $1<p<\infty$. When $p=1$ the arguments need trivial modifications that are omitted here.  
We observe   that $u_\lambda=v_\lambda$ in $R_Q=\frac{1}{10}Q\setminus S_Q$; cf.\ the proof of Lemma \ref{l.aux}.
Therefore if $p'=p/(p-1)$,   then by 
H\"older's inequality we obtain that  
\begin{align*}
\nu(Q)\lvert u_{Q;\nu}-(u_\lambda)_{R_Q;\nu}\rvert^p
&=\nu(Q)\lvert (v_\lambda)_{R_Q;\nu}-u_{Q;\nu}\rvert^p\\
&= \nu(Q)\biggl\lvert \vint_{R_Q} (v_\lambda(y)-u_{Q;\nu})\,d\nu(y)\biggr\rvert^p\\
&= \nu(Q)\biggl\lvert \vint_{R_Q}\sum_{\substack{P\in\mathcal{W}_\lambda\\y\in P^*}}
\psi_P(y)(u_{P;\nu}-u_{Q;\nu})\,d\nu(y)\biggr\rvert^p \\
&\le  \vint_{R_Q}\Bigg(\sum_{\substack{P\in\mathcal{W}_\lambda\\y\in P^*}}
\psi_P(y)^{p'}\Bigg)^{p/p'}\Bigg(\sum_{\substack{P\in\mathcal{W}_\lambda\\y\in P^*}}\nu(Q)\lvert u_{P;\nu}-u_{Q;\nu}\rvert^p\Bigg)\,d\nu(y)\,.
\end{align*}
Fix $y\in R_Q\subset Q\subset U_\lambda$. By the inequality in \eqref{e.w_u_l}, we have
\[
\Bigg(\sum_{\substack{P\in\mathcal{W}_\lambda\\y\in P^*}}
\psi_P(y)^{p'}\Bigg)^{p/p'}\le 
\Bigg(\sum_{\substack{P\in\mathcal{W}_\lambda}}
\mathbf{1}_{P^*}(y)\Bigg)^{p/p'}\le 
C(c_\nu,p)\,.
\]
Fix also $P\in\mathcal{W}_\lambda$ 
such that $y\in P^*$. Then we have $r_P/2\le r_Q\le 2r_P$ and
$P\cup Q\subset B(y_Q,270r_Q)$ for some $y_Q\in X\setminus U_\lambda$. Thus,
\begin{align*}
\nu(Q)\lvert u_{P;\nu}-u_{Q;\nu}\rvert^p
&\le C(p)\nu(Q)\bigl(\lvert u_{P;\nu}-u_{B(y_Q,270r_Q);\nu}\rvert^p+
\lvert u_{Q;\nu}-u_{B(y_Q,270r_Q);\nu}\rvert^p\bigr)\\
&\le C(c_\nu,p)\int_{B(y_Q,270r_Q)}
\lvert u-u_{B(y_Q,270r_Q);\nu}\rvert^p\,d\nu\\
&\le C(c_\nu,p)M^{\nu,p}u(y_Q)^p(540r_Q)^p\nu(B(y_Q,270r_Q))\\
&  \le C(c_\nu,p)\lambda^pr_Q^p\nu(Q)\,. 
\end{align*}
Since there are at most $C(c_\nu)$ cubes
$P\in\mathcal{W}_\lambda$ such that $y\in P^*$, we  conclude  that
\[
\sum_{\substack{P\in\mathcal{W}_\lambda\\y\in P^*}}\nu(Q)\lvert u_{P;\nu}-u_{Q;\nu}\rvert^p
\le C(c_\nu,p)\lambda^pr_Q^p\nu(Q)\,.
\]
Collecting the estimates above yields
\[
\nu(Q)\lvert u_{Q;\nu}-(u_\lambda)_{R_Q;\nu}\rvert^p
\le C(c_\nu,p)\lambda^pr_Q^p\nu(Q)\,,
\]
and this concludes  the proof of inequality \eqref{e.tast}.

We are now ready to finish the proof of the lemma. 
Recall that $B\setminus U_\lambda\not=\emptyset$.
Therefore $Q\subset 2B$ if $Q\in\mathcal{W}_\lambda$
intersects $B$.
Hence, by \eqref{e.first_est} and \eqref{e.tast},
\begin{align*}
\int_{B\cap U_\lambda} \lvert u-(u_\lambda)_{2B;\nu}\rvert^p\,d\nu
& \le C(c_\nu,p)\sum_{\substack{Q\in\mathcal{W}_\lambda\\B\cap Q\not=\emptyset}}\int_{\frac{1}{10}Q} \lvert u_\lambda-(u_\lambda)_{2B;\nu}\rvert^p\,d\nu\\
& \le C(c_\nu,p)\int_{2B} \lvert u_\lambda-(u_\lambda)_{2B;\nu}\rvert^p\,d\nu\,,
\end{align*}
where we also used the fact that 
$\{\frac{1}{10}Q\,:\,Q\in\mathcal{W}_\lambda\}$ is a pairwise disjoint family.
Finally, by taking also into account  inequality~\eqref{e.account}, we have shown that inequality \eqref{e.suffices}
holds.  
\end{proof}

The main result of this section now follows easily from the
above considerations. 

\begin{proof}[Proof of Theorem \ref{t.monotone_extension}]
The function $u_\lambda\colon X\to \R$ is defined by \eqref{e.curious_ext}.
By applying Lemma \ref{l.v_Lip} and Lemma \ref{l.b_Lip} we find that $u|_{X\setminus U_\lambda}=u_\lambda|_{X\setminus U_\lambda}$ and that $u_\lambda$
is $C\lambda$-Lipschitz, with $C=C(c_\nu,p)$. Finally, by Lemma \ref{l.monotone_extension},
we see that
$M^{\nu,p} u(x)\le C(c_\nu,p)M^{\nu,p}u_\lambda(x)$ whenever $x\in X\setminus U_\lambda$.
\end{proof}

\section{Self-improvement of global maximal  Poincar\'e  inequalities}\label{s.global}

Let $1<p<\infty$ and 
let $X=(X,d,\nu,\mu)$ be a 
 geodesic 
two-measure space.
Recall that 
$M^{\nu,p}u=M^{\nu,p}_{\mathcal{X}}u$ is the global maximal function that is 
defined with respect to the family $\mathcal{X}=\{B\subset X\,:\, B\text{ is a ball}\}$
of all balls in $X$. 

\begin{theorem}\label{t.global_improvement}
Let $1<p<\infty$.
 Assume 
that $X=(X,d,\nu,\mu)$ is a geodesic two-measure space  and 
that there is a constant $C_1>0$ such that
inequality
\begin{equation}\label{e.M_ineq}
\int_X (M^{\nu,p} u(x))^{p}\,d\mu(x)\le C_1 \int_X g(x)^p\,d\mu(x)
\end{equation}
holds whenever $u\in \Lip(X)$ and  $g\in \mathcal{D}^p(u)$.
Then there exists $0<\eps_0=\eps_0(C_1,c_\nu,p)<p-1$ 
with the property that for all $0<\eps\le \eps_0$ there is a constant 
$C=C(C_1,c_\nu,p,\eps)>0$ such that inequality
\begin{equation}\label{e.finite}
\int_X (M^{\nu,p} u(x))^{p-\eps}\,d\mu(x)\le C\int_X g(x)^{p-\eps}\,d\mu(x)
\end{equation}
holds  whenever $u\in \Lip(X)$ has a bounded support and  $g\in \mathcal{D}^p(u)$.  
\end{theorem}

\begin{proof} 
Fix $u\in \Lip(X)$ with a bounded support. Without loss of generality, we may assume that $g\in\mathcal{D}^p(u)$ is
the minimal $p$-weak upper gradient of $u$. Then $g\in L^p(X;d\mu)$ by (D3) and the  minimality of $g$.  
We  may also assume that
$M^{\nu,p}u(x)\not=0$ for every $x\in X$; otherwise
$M^{\nu,p}u(x)=0$ for every $x\in X$ and inequality \eqref{e.finite} holds.

 Fix  $\lambda>0$ and write 
$U_\lambda=\{x\in X\,:\, M^{\nu,p}u(x)>\lambda\}$.
By  Theorem \ref{t.monotone_extension},  there is
a $C(c_\nu,{p})\lambda$-Lipschitz extension $u_\lambda\colon X\to \R$ of 
$u|_{X\setminus U_\lambda}$ such that
 $M^{\nu,p} u(x)\le C(c_\nu,{p})M^{\nu,p} u_\lambda(x)$ 
whenever $x\in X\setminus U_\lambda$; recall that instead of the particular extension constructed in 
Section~\ref{s.Whitney_ext} we could use here any other 
$C(c_\nu,{p})\lambda$-Lipschitz extension satisfying the above maximal function estimate.
By condition (D3) in Section \ref{s.two_measure_PI}, we have
\[
g_\lambda =  C(c_\nu,{p})\lambda\mathbf{1}_{U_\lambda} +  g\mathbf{1}_{X\setminus U_\lambda}\in\mathcal{D}^p(u_\lambda)\,.
\]
Hence, it follows that
\begin{equation}\label{e.st}
\begin{split}
\int_{X\setminus U_\lambda} (M^{\nu,p} u(x))^p\,d\mu(x)&\le C(c_\nu,{p})^p\int_{X\setminus U_\lambda} (M^{\nu,p} u_\lambda(x))^p\,d\mu(x)\\
&\le C(c_\nu,{p})^p\int_X (M^{\nu,p} u_\lambda(x))^p\,d\mu(x)
\\&\le C(c_\nu,{p})^pC_1 \int_X g_\lambda(x)^p\,d\mu(x)\\
&\le C_2\lambda^p \mu(U_\lambda)
+ C_2\int_{X\setminus U_\lambda} g(x)^p\,d\mu(x)\,,
\end{split}
\end{equation}
where $C_2=C(C_1,c_{\nu},p)>0$.

At this stage, we consider any $0<\eps_0< p-1$, whose
value is to be fixed at the end of the proof, 
and we fix $0<\eps\le \eps_0$ and a number $t>0$. 
Multiplying both sides of \eqref{e.st} by $\lambda^{-1-\eps}$ 
and then integrating the resulting inequality from $t$ to $\infty$  
yields 
\begin{equation}\label{e.t_int}
\begin{split}
\int_t^\infty\lambda^{-1-\eps}&\int_{X\setminus U_\lambda} (M^{\nu,p} u(x))^p\,d\mu(x)\,d\lambda
\\&\le C_2
\int_t^\infty\lambda^{p-1-\eps} \mu(U_\lambda) \,d\lambda
+C_2\int_t^\infty \lambda^{-1-\eps}\int_{X\setminus U_\lambda} g(x)^p\,d\mu(x)\,d\lambda\,.
\end{split}
\end{equation}
By Fubini's theorem and integration,
\begin{align*}
\int_t^\infty\lambda^{-1-\eps}\int_{X\setminus U_\lambda} (M^{\nu,p} u(x))^p\,d\mu(x)\,d\lambda
&=\int_X (M^{\nu,p} u(x))^p \int_t^\infty  \mathbf{1}_{X\setminus U_\lambda}(x)\lambda^{-1-\eps}\,d\lambda \,d\mu(x)\\
&=\int_X (M^{\nu,p} u(x))^p \int_{\max\{t,M^{\nu,p}u(x)\}}^\infty  \lambda^{-1-\eps}\,d\lambda \,d\mu(x)
\\
&=\frac{1}{\eps}\int_X (M^{\nu,p} u(x))^{p}\max\{t,M^{\nu,p}u(x)\}^{-\eps}\,d\mu(x)\,.
\end{align*}
In a similar way, we obtain that
\begin{align*}
C_2
\int_t^\infty\lambda^{p-1-\eps} \mu(U_\lambda) \,d\lambda
&=C_2\int_X\int_t^\infty \mathbf{1}_{U_\lambda}(x)\lambda^{p-1-\eps}\,d\lambda\,d\mu(x)\\
&\le C_2\int_{\{x\in X\,:\,t<M^{\nu,p}u(x)\}}\int_t^{M^{\nu,p} u(x)} \lambda^{p-1-\eps}\,d\lambda\,d\mu(x)
\\&\le \frac{C_2}{p-\eps}\int_{\{x\in X\,:\,t<M^{\nu,p}u(x)\}} (M^{\nu,p} u(x))^{p-\eps}\,d\mu(x)\\
&\le \frac{C_2}{p-\eps}\int_X (M^{\nu,p} u(x))^{p}\max\{t,M^{\nu,p}u(x)\}^{-\eps}\,d\mu(x)\,.
\end{align*}
Finally, a similar argument and the global analogue of Lemma \ref{l.Lip_estimate} yield
\begin{align*}
C_2\int_t^\infty \lambda^{-1-\eps}\int_{X\setminus U_\lambda}g(x)^p\,d\mu(x)\,d\lambda&=
\frac{C_2}{\eps} \int_{X} g(x)^p \max\{t,M^{\nu,p}u(x)\}^{-\eps}\,d\mu(x)\\
&\le \frac{C_2C(c_\nu,\eps)}{\eps} \int_{X} g(x)^p \max\{t,g(x)\}^{-\eps}\,d\mu(x)\,.
\end{align*}
Multiplying the obtained inequalities by $\eps>0$ gives us 
\begin{align*}
\int_X&  (M^{\nu,p} u(x))^{p}\max\{t,M^{\nu,p}u(x)\}^{-\eps}\,d\mu(x)\\
&\le \frac{\eps C_2}{p-\eps}\int_X (M^{\nu,p} u(x))^{p}\max\{t,M^{\nu,p}u(x)\}^{-\eps}\,d\mu(x)+C_2C(c_\nu,\eps)\int_{X} g(x)^{p-\eps}\,d\mu(x)\,.
\end{align*}
Now we choose  $0<\eps_0=\eps_0(C_2,p)<p-1$ to be  so small that
$\eps C_2/(p-\eps)<1/2$  for all $0<\eps\le \eps_0$. Then we fix  
$0<\eps\le \eps_0$ and absorb 
the first term on the right-hand side  above  to the left-hand side.
 This absorbed term is finite by the assumed inequality~\eqref{e.M_ineq} and the fact that $g\in L^p(X;d\mu)$. 
We mention in passing that the integration
in \eqref{e.t_int} is taken with $t>0$ as a lower bound  
in order to ensure the finiteness of the absorbed term.  

Consequently, we obtain that
\begin{align*}
\int_X&  (M^{\nu,p} u(x))^{p}\max\{t,M^{\nu,p}u(x)\}^{-\eps}\,d\mu(x)  
\le 2C_2C(c_\nu,\eps)\int_{X} g(x)^{p-\eps}\,d\mu(x)\,,
\end{align*}
and the desired inequality \eqref{e.finite} follows by taking $t\to 0_+$ and using Fatou's lemma.
\end{proof}

\bibliographystyle{abbrv}
\def\cprime{$'$} \def\cprime{$'$} \def\cprime{$'$}

\end{document}